\documentclass[%
british,
final,
leqno]{article}

\usepackage[utf8]{inputenc}
\usepackage[T1]{fontenc}
\usepackage{microtype}
\usepackage{titlesec}
\usepackage{babel}
\usepackage{csquotes}

\usepackage{amsmath}
\usepackage{amsthm}
\usepackage{amssymb}
\usepackage{thmtools}
\usepackage{mathtools}
\usepackage{stmaryrd}

\usepackage{tikz-cd}
\usepackage{enumitem}
\usepackage{stackrel}
\usepackage{adjustbox}

\RequirePackage
{todonotes}

\RequirePackage{mathscinet}
\RequirePackage[backend      = biber,
                sortcites    = true,
                giveninits   = true,
                doi          = false,
                isbn         = false,
                url          = false,
                maxnames     = 5,
                maxcitenames = 5,
                minalphanames= 4,
                maxalphanames= 4,
                style        = alphabetic]{biblatex}
\DeclareNameAlias{sortname}{family-given}
\DeclareNameAlias{default}{family-given}
\DeclareFieldFormat[article]{volume}{\bibstring{jourvol}\addnbspace#1}
\DeclareFieldFormat[article]{number}{\bibstring{number}\addnbspace#1}
\renewbibmacro*{volume+number+eid}
{
    \printfield{volume}
    \setunit{\addcomma\space}
    \printfield{number}
    \setunit{\addcomma\space}
    \printfield{eid}
}
\renewbibmacro{in:}{}       
\DeclareFieldFormat{pages}{#1} 

\DeclareMathOperator{\A}{\mathbb{A}}
\DeclareMathOperator{\Z}{\mathbb{Z}}
\DeclareMathOperator{\Q}{\mathbb{Q}}
\DeclareMathOperator{\F}{\mathbb{F}}
\DeclareMathOperator{\G}{\mathbb{G}}
\renewcommand{\P}{\mathbb{P}}

\DeclareMathOperator{\X}{\mathcal{X}}
\DeclareMathOperator{\boks}{\overline{\square}}
\DeclareMathOperator{\logDM}{\mathbf{\textnormal{logDM}}^{\text{eff}}(k)}
\DeclareMathOperator{\Cor}{\textnormal{\textbf{Cor}}}
\DeclareMathOperator{\tr}{\operatorname{tr}}
\DeclareMathOperator{\MCor}{\underline{\textnormal{\textbf{M}}}\textnormal{\textbf{Cor}}}
\DeclareMathOperator{\PST}{\textnormal{\textbf{PST}}}
\DeclareMathOperator{\Shv}{\textnormal{\textbf{Shv}}}
\DeclareMathOperator{\RSC}{\textnormal{\textbf{RSC}}}
\DeclareMathOperator{\Nis}{\textnormal{Nis}}
\DeclareMathOperator{\et}{\textnormal{ét}}
\DeclareMathOperator{\HI}{\textnormal{\textbf{HI}}}
\DeclareMathOperator{\CI}{\textnormal{\textbf{CI}}}
\DeclareMathOperator{\MPST}{\underline{\textnormal{\textbf{M}}}\textnormal{\textbf{PST}}}
\DeclareMathOperator{\MNST}{\underline{\textnormal{\textbf{M}}}\textnormal{\textbf{NST}}}
\DeclareMathOperator{\Spec}{\textnormal{Spec}}
\DeclareMathOperator{\Ext}{\textnormal{Ext}}
\DeclareMathOperator{\fil}{\textnormal{fil}}
\DeclareMathOperator{\Coker}{\textnormal{Coker}}
\DeclareMathOperator{\Hom}{\textnormal{Hom}}
\renewcommand{\tilde}{\widetilde}

\renewcommand{\sup}{\text{supp }}
\renewcommand{\lim}{\text{lim }}

\declaretheoremstyle[headfont   = \bfseries\sffamily,
                     notefont   = \normalfont,
                     bodyfont   = \itshape,
                     spaceabove = 6pt,
                     spacebelow = 6pt]{plain}
\declaretheoremstyle[headfont   = \bfseries\sffamily,
                     notefont   = \normalfont,
                     spaceabove = 6pt,
                     spacebelow = 6pt]{definition}
\declaretheorem[style = plain, numberwithin = section]{theorem}
\declaretheorem[style = plain,      sibling = theorem]{corollary}

\declaretheorem[style = plain,      sibling = theorem]{proposition}

\declaretheorem[style = definition, sibling = theorem]{definition}
\declaretheorem[style = definition, sibling = theorem]{example}
\declaretheorem[style = remark,     sibling = theorem]{remark}
\declaretheorem[style = definition,     sibling = theorem]{question}
\numberwithin{equation}{section}

\usepackage{authblk}

\usepackage{varioref}
\usepackage[colorlinks,linkcolor=blue,citecolor=blue,urlcolor=teal]{hyperref}
\usepackage[capitalize]{cleveref}

\title{Lectures on the Cohomology of Reciprocity Sheaves}
\author{Nikolai Opdan and Kay Rülling}
\date{}

\begin{document}
\maketitle

\begin{abstract}
These are the notes accompanying three lectures given by the second author at the Motivic Geometry program at CAS, which 
aim to give an introduction and an overview of some recent developments in the field of reciprocity sheaves.  
We begin by introducing the theory of reciprocity sheaves and the necessary background of modulus sheaves with transfers as developed by B. Kahn, H. Miyazaki, S. Saito, and T. Yamazaki. We then explain some basic examples of reciprocity sheaves with a special emphasis on Kähler differentials and the de Rham--Witt complex. After an overview of some fundamental results, we survey the recent work of F. Binda, S. Saito and the second author on the cohomology of reciprocity sheaves. In particular, we discuss a projective bundle formula, a blow-up formula, and a Gysin sequence, which generalizes work of Voevodsky on homotopy invariant sheaves with transfers. From this, pushforwards along projective morphisms
can be constructed, which give rise to an action of projective Chow correspondences on the cohomology of reciprocity sheaves. This generalizes several constructions which originally relied on Grothendieck duality for coherent sheaves and gives a motivic view towards these results. 

We then survey some applications which include the birational invariance of the cohomology of certain classes of reciprocity sheaves, many of which were not considered before. Finally, we outline some recent results which were not part of the lecture series.
\end{abstract}

\tableofcontents

\subsection*{Foreword}
These are the notes accompanying three lectures\footnote{The lecture series can be found on YouTube at \url{https://youtube.com/playlist?list=PLiG7qomWDYXMcU-CH9ahhfWVMwR214A64}} of the second author given in October 2020
at the Motivic Geometry program at CAS, 
which aim to give an introduction and an overview of some recent developments in the
field of reciprocity sheaves.  We stress that the focus of this lecture series, and the present notes,
is on the properties of reciprocity sheaves, on their cohomology, and on applications of the theory,
with a particular emphasis on de Rham-Witt sheaves.
We do not stress categorical constructions such as the triangulated category of motives with modulus and
we do not intend to give a complete overview of the whole theory, 
which was first and  foremost developed by B. Kahn, H. Miyazaki, S. Saito, and T. Yamazaki.
We try to keep the informal style of the lectures also in these notes and we do not claim any originality. 

\subsection*{Acknowledgements}
The authors greatly acknowledge the support by the research project  ``Motivic Geometry'' at the Centre for Advanced Study at the Norwegian Academy of Science and Letters in Oslo, Norway.
We are grateful to Shuji Saito and the referee for a careful reading and many helpful comments.

\section{Reciprocity sheaves}
By work of Voevodsky and many others, the general theory of the cohomology  of \(\A^1\)-invariant sheaves with transfers is fully developed. Among the most fundamental properties are 
the projective bundle formula, the blow-up formula, the Gysin sequence, Gersten resolution, action of proper Chow correspondences, representability of cohomology theories, etc...

However, the theory has a drawback: Many interesting non-\(\A^1\)-invariant sheaves share the same properties as above, such as Kähler differentials, smooth commutative unipotent group schemes, étale motivic cohomology with \(\Z/p^n\)-coefficients (in char \(p>0\)), etc. Despite this, they are not representable in the classical motivic theory and hence cannot be studied by motivic methods. 
This is in part because the \(\A^1\)-invariant theory only detects log poles, regular singularities and tame ramifications. In order to study more general theories, we need a more general theory than the classical theory provided by Voevodsky.

One approach was recently introduced by Binda--Park--Østvær in \cite{BPO}. The basic idea is to generalize the classical motivic homotopy theory by replacing  \(\A^1\) with the ``cube''
\[\boks := (\P^1, \infty).\]
It is the log scheme whose underlying scheme is \(\P^1\) and whose log structure is induced by the inclusion of the divisor \(\infty \hookrightarrow \P^1\). Working with log smooth log schemes and a suitable topology, the authors construct in {\em loc. cit.} the triangulated category \(\logDM\) of effective logarithmic motives. A cohomology theory representable in \(\logDM\) has the nice properties listed above (at least under the assumption of the existence  of resolutions of singularities). An example of such a theory is the sheaf of log-Kähler differentials which becomes representable in this new triangulated category. Embedding the classical triangulated category of motives \(\text{DM}^{\text{eff}}(k)\) fully faithfully in \(\logDM\), they construct an enlargement of the classical \(\A^1\)-invariant theory. So far there is however no pole order or ramification filtration on the sheaves in this category.

The theory of reciprocity sheaves (\cite{BSY}, \cite{KSY-RecII}) provides another solution: The basic idea, which goes back to B. Kahn in the 1990's, 
is to consider only those sheaves whose sections behave in a controlled way at infinity, i.e., replace \(\A^1\)-invariance by a \textit{modulus condition}. This is a similar condition to the one considered by Rosenlicht--Serre to define the generalized Jacobian for curves.

\subsection{Modulus à la Rosenlicht and Serre}\label{subsec:RoSe}
The definition of a modulus condition goes back to Rosenlicht  \cite{Rosenlicht} and Serre \cite[III,][]{Ser}. They considered the modulus of a rational map from a curve to a commutative algebraic group.
\begin{definition}
Let \(k\) be a perfect field, \(C\) a smooth projective curve over \(k\) with an effective divisor \(D\), \(U := C \setminus  D \) the complement of \(D\), and \(G\) a smooth commutative \(k\)-group. Then a $k$-morphism  \(a \colon U \to G\) has \textit{modulus} \(D\) if 
\[\sum_{x \in U} v_x(f) \cdot \text{Tr}_{x/k}(a(x))=0,\]
for all \(f \in k(C)^{\times}\) with \(f \equiv 1 \mod D\), i.e., 
$f\in \bigcap_{x\in D}{\rm Ker}( \mathcal{O}_{X,x}^\times\to \mathcal{O}_{D,x}^\times)$,
where $v_x$ denotes the discrete valuation defined by the point $x$ and 
$\text{Tr}_{x/k}:G(x)\to G(k)$ is the trace\footnote{If $G=\G_a$ is the additive group, then this is the usual trace, if
$G=\G_m$ is the multiplicative group, then  this is the norm.}, e.g. \cite[Exp XVII, App 2]{SGA4III}.
\end{definition}

The choice of a rational point $x\in U(k)$ gives a universal morphism called  the Albanese map
\[{alb}_{(C,D)}\colon U \to \text{Alb}(C, D),\]
with the property that any map \(a \colon U \to G\) to a smooth commutative group scheme \(G\), which satisfies the modulus condition, factors via ${\rm alb}_{(C,D)}$.

Let \({\rm Sm}\) denote the category of  smooth separated $k$-schemes of finite type, in the following
simply called  smooth $k$-schemes. For $X$, $Y\in {\rm Sm}$ denote by $\Cor(X,Y)$ the group of finite correspondences  from $X$ to $Y$ as introduced by Suslin-Voevodsky,  i.e., it is the free abelian group generated by integral closed subschemes of $X\times Y$, which are finite and surjective over a connected component of $X$. There is a category of finite correspondences $\Cor$ whose objects are the objects of ${\rm Sm}$ and with morphisms the finite correspondences. A {\em presheaf with transfers} is an additive contravariant functor from $\Cor$ to the category of abelian groups.
The  category of presheaves with transfers is denoted by $\PST$.

Using the trace construction alluded to above, any smooth commutative $k$-group admits a structure of a presheaf with transfers by means of which 
the modulus condition may be reformulated as follows:
\begin{definition}\label{def:RoSe-Cor}
An element \(a \in G(U)\) has \textit{modulus} \(D\) if \[\gamma^* a = 0,\]
for all prime correspondences \(\Gamma \in \Cor(\P^1 \setminus \{1\}, U)\) such that
\[\{1\}\vert_{\Gamma^N} \geq D\vert_{\Gamma^N},\]
where \(\Gamma^N \to \P^1 \times C\) is the normalization of the closure of \(\Gamma\), and
\[\gamma := i^*_0 \Gamma - i^*_{\infty}\Gamma \in \Cor(\Spec k, U).\]
\end{definition}

\subsection{Modulus pairs}
The framework of modulus presheaves with tranfers was introduced by Kahn-Miyazaki-Saito-Yamazaki in \cite{KMSY1} and \cite{KMSY2}, 
as a fundamental tool in the construction of their triangulated category of motives with modulus.
\begin{definition}[{\cite{KMSY1}}]
Fix a perfect field \(k\).
A \textit{modulus pair} \(\X\) is a pair \((X, D)\), where \(X\) is a separated scheme of finite type over \(k\) and \(D\) is an effective Cartier divisor (or the empty scheme) on \(X\) such that the complement of the support of \(D\) in \(X\) is smooth. 
A modulus pair \((X, D)\) is \textit{proper} if \(X\) is proper over \(\Spec k\).

The group of \textit{modulus correspondences} from \((X, D)\) to \((Y, E)\), denoted \(\MCor((X,D), (Y, E))\), is the subgroup of \(\Cor(X\setminus D, Y\setminus E)\) generated by finite prime correspondences
\[V \subset (X \setminus  D ) \times (Y \setminus E )\] 
such that
\begin{enumerate}[label=(\roman*)]
    \item\label{MPi} The projection \(V^N \to X\) is proper,
    \item\label{MPii} \(D\vert_{V^N} \geq E\vert_{V^N}\),
\end{enumerate}
where \(V^N \to X \times Y\) is the normalization of the closure of \(V\) in $X\times Y$. An element of 
\[\MCor((X,D), (Y, E))\]
is called a \textit{modulus correspondence} from \((X, D)\) to \((Y, E)\).

The composition of finite correspondences restricts to a composition of modulus correspondences\footnote{We remark that condition \ref{MPi} is essential for this, see \cite[Prop 1.2.3]{KMSY1}.}. 
Hence we can define the category \(\MCor\) as the \textit{category of modulus pairs}, i.e., objects are modulus pairs and morphisms are modulus correspondences.
\end{definition}

The category \(\MCor\) has a monoidal structure given by
\[(X, D) \otimes (Y, E) = (X \times Y, p^*_X D + p^*_Y E),\]
where \(p_X\) (resp. \(p_Y\)) denotes the projection onto the first (resp. second) factor.

Let \(\MPST\) denote the category of presheaves of abelian groups on \(\MCor\). It comes with a  monoidal structure \(\otimes_{\MPST}\) 
which via the Yoneda embedding extends the one on \(\MCor\).
There is an adjoint functor pair
\[\underline{\omega}_{!} : \MPST \rightleftarrows \PST : \underline{\omega}^*\]
such that
\(\underline{\omega}_!G(X)=G(X, \emptyset)\) and \(\underline{\omega}^*F(X, D)=F(X \setminus  D )\).

The modulus presheaf represented by a modulus pair \((X, D)\) in \(\MPST\) is denoted by
\[\Z_{\tr}(X, D):= \MCor(-, (X, D)).\]

Let \(\boks\) denote the modulus pair \((\P^1, \infty)\), and set
\[h_0^{\boks}(X, D) := \Coker\left(\Z_{\tr}(X, D)(- \otimes \boks) \overset{i^*_0 - i^*_1}{\longrightarrow}\Z_{\tr}(X, D)\right),\]
which we can consider as the \textit{cubical modulus} version of \(h_0\) of the Suslin complex. 

\begin{remark}\label{rmk:surj}
Note that we have surjections 
\[\Z_{\tr}(X\setminus  D  ) \twoheadrightarrow \underline{\omega}_! h_0^{\boks}(X, D) 
\twoheadrightarrow h_0^{\A^1}(X \setminus  D ),\]
where $\Z_{\tr}(X\setminus  D  )=\Cor(-, X\setminus D)$ and 
\[h_0^{\A^1}(X \setminus  D )=
\Coker \left(\Z_{\tr}(X\setminus  D  )(-\otimes \A^1)
\overset{i^*_0-i^*_1}{\longrightarrow} Z_{\tr}(X\setminus  D)\right)\]
is its maximal $\A^1$-invariant quotient.

Indeed, the surjectivity of the first  map follows from the right exactness of $\underline{\omega}_!$  and the equality
\[\underline{\omega}_!\Z_{\tr}(X,D)=\Z_{\tr}(X\setminus D).\]
To see this, observe that if $V\in \Cor(S, X\setminus D)$ is a finite prime correspondence,
then  $V$ is already closed in \(S \times X\)
since it is finite over \(S\). Hence $V\in\MCor((S,\emptyset), (X,D))$ . 
The second surjection follows directly from the fact that \(\Z_{\tr}(X, D)(- \otimes \boks)\) is a subpresheaf of 
\(\Z_{\tr}(X\setminus D)(- \otimes \A^1)\).

We remark that by the above we have for $S\in {\rm Sm}$
\[\underline{\omega}_! h_0^{\boks}(X, D)(S)=\Coker\left(
\MCor((S,\emptyset)\otimes \boks, (X,D))\overset{i_0^*-i_1^*}{\longrightarrow} 
\Cor(S, X\setminus D)\right).\]
\end{remark}

\begin{definition}[{\cite[Def 2.2.4]{KSY-RecII}}]\label{def:ModulusCondition}
Let \((X, D)\) be a proper modulus pair with \(U:= X \setminus  D\), \(F\) a presheaf with transfers, and 
\(a \in F(U)\) a  section. We say that \(a\) has \textit{modulus} \((X, D)\) if 
the Yoneda map defined by \(a\) factors through \(\underline{\omega}_! h_0^{\boks}(X, D)\), i.e., there exists a map that makes the following diagram commute
\begin{center}
    \begin{tikzcd}
    \Z_{\tr}(U) \arrow[rr, "a" above, "\text{Yoneda}" below] \arrow[rd, two heads]&& F\\
    &\underline{\omega}_! h_0^{\boks}(X, D) \arrow[ur, dotted].
    \end{tikzcd}
\end{center}
\end{definition}

\begin{remark}
\begin{enumerate}
    \item In \cite{KSY-RecII}  the pair $(X,D)$ above is called an SC-modulus of $a$ in order to distinguish it from a slightly different
    notion of modulus which was introduced before in \cite{BSY}. In \cite[Thm 3.2.1]{KSY-RecII} it is proven, that the two 
    notions of modulus coincide as long as  $X\setminus D$ is quasi-affine.
    In the following we will only work with the above definition of modulus and therefore simply say {\em modulus} instead of {\em SC-modulus}.
    \item If $F=G$ is a smooth commutative $k$-group and  $X$ is a smooth projective curve, then evaluating the diagram above at \(k\) 
    gives back the definition of modulus  introduced by Rosenlicht-Serre which was reformulated in \Cref{def:RoSe-Cor}.
\end{enumerate}
\end{remark}

\subsection{Reciprocity sheaves}
We are now in the position to define reciprocity sheaves.
\begin{definition}[{\cite[Def 2.2.4]{KSY-RecII}}]\label{def:RSC}
We say that a presheaf with transfers \(F\) is a \textit{reciprocity presheaf} if for any smooth $k$-scheme \(U\), and for all \(a \in F(U)\), there exists a proper modulus pair \((X, D)\) such that $U=X\setminus D$ and \(a\) has modulus \((X, D)\). 

We let \(\RSC\) denote the full subcategory of \(\PST\) consisting of reciprocity presheaves. 
The category of \textit{reciprocity sheaves} is
\(\RSC_{\Nis}:=\RSC \cap \textbf{NST}\), where $\textbf{NST}$ denotes the subcategory of $\PST$ consisting of presheaves with transfers which are Nisnevich sheaves on ${\rm Sm}$.  
\end{definition}

\begin{remark}
\begin{enumerate}
    \item In {\em loc. cit.} the term {\em presheaves with transfers with SC-reciprocity} is used for what above is called reciprocity presheaf
in order to distinguish them from the reciprocity presheaves introduced in \cite{BSY}. 
This difference is however not relevant for us since we  mostly work with Nisnevich sheaves with transfers, for which the two notions coincide, see \cite[Cor 3.2.3]{KSY-RecII}.
\item A precursor of reciprocity sheaves are the reciprocity functors  defined in \cite{IR}.
These are defined only on function fields and regular curves over such. 
It follows from  the injectivity Theorem \cite[Thm 6]{BSY} that the restriction of any 
reciprocity sheaf in the sense of Definition \ref{def:RSC} 
to fields and regular curves defines a reciprocity functor, see \cite[Thm 5.7]{RSY}.
\end{enumerate}

\end{remark}

\begin{example}
Smooth commutative \(k\)-groups and homotopy invariant Nisnevich sheaves provide important examples of reciprocity sheaves. For the homotopy invariant sheaves this follows directly from Remark \ref{rmk:surj};
for the smooth commutative groups  the argument is essentially given by Rosenlicht-Serre\footnote{The argument of Rosenlicht-Serre works for curves over an algebraically closed field, see \cite[Thm 4.1.1]{BSY} for an extension of the argument to the general case.},
see \cite[Chap III, Thm 1]{Ser}.
\end{example}

\begin{proposition}\label{prop:KahlerReciprocity}
The (absolute) Kähler differentials are reciprocity sheaves, i.e., 
\[\Omega^j \in \RSC_{\Nis}, \quad \text{for all } j \geq 0.\]

\end{proposition}
\begin{proof}
    As a corollary of \cite[Thm 3.1.8]{CR11} Kähler differentials have an action of finite transfers. 
    It remains to show that they satisfy the modulus condition in \cref{def:ModulusCondition}. 
    
    For any form \(a \in \Omega^j(U)\) there is a proper modulus pair \((X, D)\) such that 
    \(a \in H^0(X, \Omega^j_{X} \otimes_{\mathcal{O}_{X}} \mathcal{O}_{X}(D))\). 
    We claim that  $(X,2D)$ is a modulus for $a$ in the sense of Definition \ref{def:ModulusCondition}.
    For an integral smooth $k$-scheme $S$ the restriction $\Omega^j_S\to \Omega^j_{k(S)}$ is injective, hence we reduce to show the following:
    If $C$ is a regular projective curve \(C\)  over a $k$-function field \(K\) which comes with a map to  \(X\), such that its image is not contained in $D$, 
    and  \(f \in K(C)\) satisfies the condition \(f \equiv 1 \mod 2D_C\), where $D_C$ denotes the pullback of $D$ to $C$,
    then we have to show 
    \begin{equation}\label{prop:KahlerReciprocity1}
    {\rm div}_C(f)^*a=0.
    \end{equation}
    To this end, observe that the modulus condition for $f$ and the choice of $(X,D)$ imply
    \[\text{Res}_x(a_{|C} \text{ dlog}(f)) = 0,\]
    for all \(x \in  D_{ C} \), where $\text{Res}_x:\Omega^{j+1}_{K(C)}\to \Omega^j_K$ denotes the residue symbol.
    Since the pullback of $a$ to $C \setminus D_{C}$ is regular, we have 
    \[\text{Res}_x(a_{|C} \text{ dlog}(f))= v_x(f) \tr_{K(x)/K}(a(x)),\quad \text{for all }x\in C \setminus D_{ C},\]
    with $\tr_{K(x)/K}: \Omega^j_{K(x)}\to \Omega^j_K$ the trace.
    Thus the reciprocity law yields
    \[0 = \sum_{x \in C} \text{Res}_x(a \text{ dlog}(f))= \sum_{x\in C \setminus D_{ C} }v_x(f) \tr_{K(x)/K}(a(x)),\]
    which is a reformulation of  \eqref{prop:KahlerReciprocity1}, cf. subsection \ref{subsec:RoSe}.
    
\end{proof}

\begin{definition}
For a reciprocity presheaf \(F \in \RSC\) we form the \textit{modulus presheaf} \(\tilde{F}\) by defining
\[\tilde{F}(X, D) := \{a \in F(X \setminus  D) : a \text{ has modulus } (\overline{X}, \overline{D}+N \cdot B), \text{ for some } N \gg 0\},\]
where \((\overline{X}, \overline{D}+  B)\) is a compactification of \((X, D)\), in the sense of \cite[Def 1.8.1]{KMSY1}, i.e.,
it is a proper modulus pair with $X=\overline{X}\setminus B$ and $D= \overline{D}_{|X}$.
This definition is independent of the choice of the compactification, by e.g., \cite[Rmk 1.5]{Sai}.
\end{definition}

We have  \(\tilde{F} \in \MPST\) and  it satisfies 
\begin{itemize}
    \item \(\boks\)-\textit{invariance}: \(\tilde{F}(\X \otimes \boks) \cong \tilde{F}(\X)\),
    \item M-\textit{reciprocity}: \(\tilde{F}(X, D) = \varinjlim_{N} \tilde{F}(\overline{X}, \overline{D} + N \cdot B)\),
    \item \textit{Semi-purity}: \(\tilde{F}(X, D) \subset \tilde{F}(X \setminus  D , \emptyset)\).
\end{itemize}
In fact,  $\tilde{F}$ is a $\boks$-invariant modulus presheaf by \cite[Prop 2.3.7]{KSY-RecII} and \cite[Prop 2.4.1]{KSY-RecII},
and the other two properties, follow directly from the definition.

\begin{definition}
We define \(\CI^\tau\) as the full subcategory of \(\MPST\) consisting of presheaves satisfying cube-invariance and M-reciprocity.
We let \(\CI^{\tau, sp}\) denote the full subcategory of \(\CI^\tau\) consisting of semi-pure presheaves.
\end{definition}

This gives an adjoint functor pair (cf. \cite[Prop 2.3.7]{KSY-RecII})
\[\underline{\omega}_! : \CI^{\tau, {\rm sp}} \rightleftarrows \RSC : \underline{\omega}^{\CI},\]
where \(\underline{\omega}^{\CI}(F) = \tilde{F} \in \CI^{\tau, sp}\).
For a  proper modulus pair \(\X\) we obtain \(\underline{\omega}_!h_0^{\boks}(\X) \in \RSC\),
see \cite[Cor 2.3.5]{KSY-RecII}.

\subsection{Modulus sheaves}
\begin{definition}\label{defn:mod-sheaf}
A presheaf \(G \in \MPST\) is a \textit{modulus sheaf} if for all modulus pairs \(\X = (X, D)\)
the presheaf on  the category of \'etale $X$-schemes
\[(U \overset{\et}{\to} X) \mapsto G(U, D_{\vert U}) =: G_{\X}(U)\]
is a Nisnevich sheaf.

We let \(\MNST\) denote the category of Nisnevich modulus sheaves. Note that \(\underline{\omega}_!\) restricts to a functor \(\MNST \to \textbf{NST}\) which we also denote by \(\underline{\omega}_!\).
\end{definition}

\begin{remark}
Note that \(\MNST\) is not the category of sheaves on a site whose underlying category is $\MCor$.
However, there is a subcategory of $\MCor$ which underlies a site associated to a regular and complete
cd-structure, such that \(G \in \MPST\) is a sheaf in the above sense if and only if the restriction 
of $G$ is a sheaf on this site, see \cite[Prop 3.2.3]{KMSY1}.
\end{remark}

\begin{theorem}[\cite{KMSY1},\,\cite{KMSY2}]
The natural inclusion $\MNST\to \MPST$ admits an exact left adjoint, the so-called sheafification,
\[\underline{a}_{\Nis} \colon \MPST \to \MNST,\]
which sends  presheaves with M-reciprocity to sheaves with M-reciprocity.
\end{theorem}
The existence of the sheafification functor is proven in \cite{KMSY1}. It follows from \cite[Thm 2]{KMSY2} that it is compatible with 
\(M\)-reciprocity. By \cite[Thm 2]{KMSY1} the sheafifcation functor is determined by the formula
\[\underline{a}_{\Nis}(G)_{(X, D)} = \varinjlim f_*(G_{(Y, f^*D), \Nis}),\]
where the colimit is taken over a directed set of proper morphisms \(f \colon Y \to X\) which induce an isomorphism on the complement
\(Y \setminus  f^*D \cong X \setminus  D  \) and the index $\Nis$ on the right denotes the Nisnevich
sheafification on the category of \'etale $Y$-schemes.
Thus, 
\begin{equation}\label{eq:NisNis}
\underline{\omega}_!(\underline{a}_{\Nis}(G)) = (\underline{\omega}_!G)_{\Nis},
\end{equation}
where the index $\Nis$ on the right denotes the Nisnevich sheafification on ${\rm Sm}$,
which by a result of Voevodsky restricts to a functor $\PST\to\textbf{NST}$.
Furthermore,
\begin{equation}\label{CorKMSY}
  \Ext^i_{\MNST}(\Z_{\tr}(X, D), G) = \varinjlim H^i(Y_{\Nis}, G_{(Y, f^*D)}).  
\end{equation}

This leads to the following question, see \cite[Question 1]{KMSY1}:
\begin{question}
Does \eqref{CorKMSY} stabilize for 
\(G \in \CI^{\tau, sp}_{\Nis}:= \CI^{\tau, sp} \cap \MNST?\)\footnote{It is shown in \cite[no 6.9]{RS-AS} that this question has a negative answer,
if the base field has positive characteristic $p$, the divisor of the modulus pair $(X,D)$ 
has a component of multiplicity divisible by $p$ and $G=\widetilde{\Omega^q}$ with $q\neq 0,\dim X$.  
But it remains an interesting question for which $G$ and $(X,D)$ one has a positive answer. 
Assuming resolutions of singularity this is for example the case if $q=\dim X$, see \cite[Cor 7.5]{RS-AS}.}
\end{question}

A fundamental result by S. Saito generalizes Voevodsky's strict homotopy invariance theorem by proving that Nisnevich sheafication preserves \(\boks\)-invariance. The proof requires in particular a delicate extension of  Voevodsky's theory of standard triples 
(see \cite[Chap 3, 4.]{FV}) to the setup of modulus pairs. 
\begin{theorem}[{\cite[Thm 10.1]{Sai}}]\label{thm:NisSheafication}
We have
\[\underline{a}_{\Nis}(\CI^{\tau, sp}) \subset \CI^{\tau, sp}_{\Nis}.\]
\end{theorem}


\begin{corollary}\label{cor:RSC-ab}
For every \(F \in \RSC\), the Nisnevich sheafification \(F_{\Nis}\) belongs to \(\RSC_{\Nis}\).
In particular, \(\RSC_{\Nis} \subset \textnormal{\textbf{NST}}\) is a full abelian subcategory.
\end{corollary}
\begin{proof}
By \cref{thm:NisSheafication} we have
\[G := \underline{a}_{\Nis}(\tilde{F}) \in \CI^{\tau, sp}_{\Nis}=\CI^{\tau, sp} \cap \MNST.\]
Together with \eqref{eq:NisNis} we get
\[F_{\Nis} = \underline{\omega}_!(G) \in \RSC\cap\textbf{NST}= \RSC_{\Nis}.\]
\end{proof}

The following purity theorem by S. Saito 
generalizes the \(\A^1\)-invariant purity theorem by Voevodsky and will be essential in
section \ref{sec:CohRec}.
\begin{theorem}[{\cite[Theorem 0.2]{Sai}}] \label{thm:ReciprocityPurity}
For \(F \in \RSC_{\Nis}\) and $x \in X^{(c)}$ a $c$-codimensional point  in \(X\)  we have
\[H^i_x(X, F)=0,\quad \text{for } i \neq c,\]
and
\begin{equation}\label{eq:DependsOnC}
    H^c_x(X, F) \simeq F_{-c}(X):= \frac{F((\A^1 \setminus \{0\})^c \times x)}{\sum_{i=1}^n F((\A^1 \setminus \{0\})^{i-1} \times \A^1 \times (\A^1 \setminus \{0\})^{c-i} \times x)}.
\end{equation}
\end{theorem}

\begin{remark}
The isomorphism in \eqref{eq:DependsOnC} depends on the choice of a $k$-isomorphism 
\[k(x)\{t_1, \hdots, t_c\} \overset{\simeq}{\longrightarrow} \mathcal{O}^h_{X, x},\]
where \(\mathcal{O}^h_{X, x}\) denotes the henselization of \(\mathcal{O}_{X, x}\).
In particular, it is not functorial. 
(As one directly sees by considering the case $F=\G_a$ and $c=1$.) 
Note that if $F$ is $\A^1$-invariant the isomorphism is independent of this choice,
see \cite[Chap. 3, Lem 4.36]{FV}.
\end{remark}

Using \cref{thm:ReciprocityPurity} we find that the $E_1$-complex of the coniveau spectral sequence has vanishing cohomology except in degree zero, hence we obtain the  \textit{Cousin resolution}
\begin{equation}\label{eq:Cousin}
0 \to F_X \to \bigoplus_{x \in X^{(0)}} i_{x*} H^0_x(F) \to \cdots \to \bigoplus_{x\in X^{(c)}} i_{x*} H_x^c(F) \to \cdots,
\end{equation}
on $X_{\Nis}$ as a generalization of the  Gersten resolution in the $\A^1$-invariant case.
The injectivity of the first morphism was already proven in \cite{BSY}.

\subsection{Relation with the logarithmic theory}
Further results by S. Saito give the relation with the logarithmic theory of Binda--Park--Østvær.

\begin{definition}\label{defn:transversal}
We say that \(\X = (X, D)\) is an \textit{ls modulus pair} if \(X \in {\rm Sm}\) and \( D_{\rm red}\) is a strict normal crossing divisor on \(X\). 
(Note that $D$ is allowed to be non-reduced.)
Let \(\MCor_{ls}\) denote the full subcategory of \(\MCor\) of ls modulus pairs.

A morphism \(f \colon Y \to X\) of smooth schemes is \textit{transversal} to \(D\) if 
\[f^{-1}(D_1 \cap \hdots \cap D_r) \hookrightarrow Y\]
is regular closed embedding of codimension equal to \(r\), for any irreducible components \(D_1, \hdots, D_r\) of \(D_{\rm red}\). 
\end{definition}

\begin{definition}
Let \((X, \mathcal{M})\) be a smooth log smooth scheme, 
where $\mathcal{M}$ is a monoid sheaf with a multiplicative map $\mathcal{M}\to \mathcal{O}_X$, 
which is an isomorphism over $\mathcal{O}_X^\times$, defining the log structure. 
By definition $\sup\mathcal{M}$ denotes 
the support of the monoid sheaf $\mathcal{M}/\mathcal{O}_X^\times$.
We have \((X, \sup \mathcal{M})\in \MCor_{\text{ls}}\) (e.g. \cite[Lem A.5.10]{BPO}) 
and define for \(F \in \RSC_{\Nis}\) 
\[F^{\log}(X, \mathcal{M}) :=  \tilde{F}(X, \sup\mathcal{M}).\]
\end{definition}

\begin{theorem}[{\cite[Thm 6.1 and Thm 6.3]{Sailog}}]
Let \(\Shv_{\text{dNis}}^{\text{ltr}}\) denote the category of dividing Nisnevich sheaves with log-transfers on log smooth fs log schemes, in the sense  of \cite[Def 4.2.1]{BPO}.
Then there exists a functor
\[\RSC_{\Nis} \longrightarrow \Shv_{\text{dNis}}^{\text{ltr}}\]
sending \(F \mapsto F^{\log}\), which is exact and fully faithful.
Moreover, \(F^{\log}\) is strictly \(\boks\)-invariant, that is for smooth log smooth schemes 
$(X,\mathcal{M})$ we have 
\[H^i_{\text{dNis}}((X, \mathcal{M}), F^{\log}) = H^i_{\text{dNis}}((X, \mathcal{M}) \times \boks, F^{\log})\]
and the Nisnevich cohomology of $F_{(X,\sup{\mathcal{M}})}$ 
(with the notation from Definition \ref{defn:mod-sheaf}) is representable
in the triangulated category of logarithmic motives \(\logDM\) constructed in \cite{BPO}:
\[H^i(X, F_{(X,\sup{\mathcal{M}})}) \cong \Hom_{\logDM}(M(X, \mathcal{M}), F^{\log}[i]).\]
\end{theorem}

\section{De Rham--Witt sheaves as reciprocity sheaves}
In this section we give a short introduction to the de Rham--Witt sheaves introduced 
by Bloch \cite{Bloch} and Deligne-Illusie \cite{Illusie}, describe some basic properties, 
define transfers, and show that they are reciprocity sheaves. From this one obtains 
many interesting reciprocity sheaves, which are useful, e.g.,  in the 
study of crystalline cohomology, the Brauer group, and \'etale motivic cohomology with $p$-primary 
torsion coefficients.

\subsection{Motivation}
Let \(X\) be a smooth projective scheme over \(\F_{p^n}\) and set
\[H^i:= H^i_{\text{crys}}(X/W(\F_{p^n}))\left[\tfrac{1}{p}\right],\]
the $i$-th  crystalline cohomology group, 
where \(W(\F_{p^n})\) denotes the Witt vectors of \(\F_{p^n}\).
Then $H^i$ is a finite dimensional vector space over  the field $W(\F_{p^n})[1/p]$,
and considering the action of Frobenius on it, it becomes a \(F := (F^n_X)^*\)-crystal. Such crystals have a slope decomposition
\[H^i \cong \bigoplus H^i_\lambda\]
where \(\lambda\) ranges through the non-negative rational numbers. Here \(H^i_{\lambda}\) is a sub-vector space on which the Frobenius acts with eigenvalues having \(p\)-adic valuation equal to \(\lambda\).

One of the main motivation behind the construction of the de Rham--Witt complex 
is the wish  to understand this slope decomposition from a cohomological point of view. 
And indeed, the de Rham-Witt complex $W\Omega^*$ computes crystalline cohomology by
\[H^i= H^i(X_{\text{Zar}}, W\Omega^*) \left[\tfrac{1}{p}\right].\]
Furthermore, the Hodge-to-de Rham spectral sequence yields 
the \textit{slope spectral sequence} given by
\[E^{j, i}_1 = H^i(X, W\Omega^j) \left[\tfrac{1}{p}\right]\Longrightarrow H^*,\]
which degenerates to give
\[\bigoplus_{j \leq \lambda < j+1} H^i_\lambda = H^{i-j}(X, W\Omega^j) \left[\tfrac{1}{p}\right].\]

\subsection{Witt vectors}
Let \(A\) be an \(\F_p\)-algebra. We recall, that for \(n\geq 1\) 
the  \textit{Witt vectors} of length  \(n\) of  \(A\) form a ring whose underlying set
is equal to \(W_n(A) = \{(a_0, \hdots, a_{n-1}) : a_i \in A\}\),
and whose ring structure is defined in such a way that 
the following properties hold:
The map  
\[R : \left\{
\begin{matrix}
W_{n+1}(A)& \longrightarrow & W_{n}(A)\\
(a_0, \hdots, a_{n}) & \longmapsto & (a_0, \hdots, a_{n-1}),
\end{matrix}
\right.
\]
is a ring-map, called the {\em restriction}.
The map 
\[
F \colon \left\{
\begin{matrix}
W_{n+1}(A) & \longrightarrow & W_{n}(A)\\
(a_0, \hdots, a_{n}) & \longmapsto & (a_0^p, \hdots, a_{n-1}^p),
\end{matrix}
\right.
\]
is a ring-map, called the {\em Frobenius}. The map
\[
V \colon \left\{
\begin{matrix}
W_{n}(A) & \longrightarrow & W_{n+1}(A)\\
(a_0, \hdots, a_{n-1}) & \longmapsto & (0, a_0, \hdots, a_{n-1}),
\end{matrix}
\right.
\]
is a group-map, called the {\em Verschiebung} (or shift).
The map
\[
[-] : \left\{
\begin{matrix}
A & \longrightarrow & W_{n}(A)\\
a & \longmapsto & [a]:=(a, 0, \hdots, 0),
\end{matrix}
\right.\]
is multiplicative and is called the Teichmüller lift.
Furthermore,
\begin{itemize}
    \item \(W_1(A) = A\) as a ring,
    \item \((a_0, \hdots, a_{n-1}) = \sum_{i=0}^{n-1}V^i([a_i])\), 
           where \(V^i=\underbrace{V\circ\ldots\circ V}_{i-\text{times}}\),
    \item \(FV=VF\) is multiplication by \(p\),
    \item \(V(a) \cdot b = V(a \cdot F(b))\).
\end{itemize}
Passing to the limit 
\[W(A) := \underset{\underset{n}{\longleftarrow}}{\lim} W_n(A),\]
where the transition maps are given by the restriction,
we get a ring which is \(p\)-torsion free if \(A\) is reduced.
For details, see, e.g., \cite[Chap II, \S6]{Serre-LF}.

\begin{example}
The Witt-vectors have the following properties:
\begin{itemize}
    \item \(W(\F_p) = \underset{\underset{n}{\longleftarrow}}{\lim}  W_n(\F_p) = \Z_p.\)
    \item if \(A\) is perfect, i.e., the Frobenius is an isomorphism, then \(W_n(A)\) is the unique flat \(\Z/p^n\Z\)-lift of \(A/\F_p\).
    \item The contravariant functor 
    \[ ({\rm Schemes}/\F_p)^{o}\to (\text{Ab-groups}), \quad X\mapsto H^0(X, W_n\mathcal{O}_X),\]
    is represented by a ring scheme \(W_n\). Thus the ring $W_n(A)$ is equal to the $A$-rational points of $W_n$.
    \item Any commutative unipotent \(\F_p\)-group scheme can be embedded into \(\bigoplus_{n_i} W_{n_i}\), viewed as a group scheme.
\end{itemize}
\end{example}

\subsection{The de Rham--Witt complex}
The {\em de Rham-Witt complex} of an \(\F_p\)-scheme \(X\),
as defined by Bloch--Kato \cite{Bloch},  \cite{Kato-LCFII},
or Deligne--Illusie \cite{Illusie} is a pro-differential graded algebra
\[((W_n \Omega^*, d)_{n\geq 1}, R),\]
where \(d\) is the differential and \(R\) is the restriction map, such that 
\[W_n \Omega^0=W_n \mathcal{O}_X.\] 
Furthermore, it is equipped with an extension of the Frobenius map 
\[F \colon W_{\bullet+1}\Omega^*_X \to W_{\bullet}\Omega^*_X,\]
and an extension of the Verschiebung map on \(W_{\bullet}\mathcal{O}_X\)
\[V \colon W_{\bullet}\Omega^*_X \to W_{\bullet+1}\Omega^*_X,\]
which satisfy the following conditions:
\begin{itemize}
    \item $F$ is a map of graded rings, $V$ is a map of graded groups,
    \item the composition of maps \(F V\) is given by multiplication by \(p\),
    \item \(V(\alpha) \cdot \beta = V(\alpha \cdot F(\beta))\) (\(F\)-linearity),
    \item \(F d V=d\),
    \item \(F d[a] = [a]^{p-1}d[a]\), for \(a\in W_\bullet\mathcal{O}_X\).
\end{itemize}

In fact, in \cite{HM} Hesselholt-Madsen show that  \(W_{\bullet}\Omega^*_X\) 
is the initial object in the category of pro-differentially graded algebras with the above 
properties. They also extend the definition to all $\Z_{(p)}$-algebras.

Moreover,  we have \(W_1\Omega^*_X = \Omega^*_{X/\F_p},\) and  there is a commutative square
\begin{equation*}
    \begin{tikzcd}
        W_{n+1}\Omega^j_X \arrow[r, "F"] \arrow[d] & W_n \Omega^j_X \arrow[d] \\
        \Omega^j_X \arrow[r, "C^{-1}"] & \Omega^j_X/d \Omega^{j-1}_X,
    \end{tikzcd}
\end{equation*}
i.e., \(F\) lifts the inverse Cartier operator $C^{-1}$, which is determined by the formula
\(C^{-1}(a d\log b) = a^p d\log b\).

\begin{remark}
\begin{enumerate}
    \item For an $\F_p$-algebra $A$, Bloch constructed the de Rham-Witt complex in \cite{Bloch} as 
    the pro-object
\[W_{\bullet}\Omega^q_A \cong 
T S \ker\left( K_{q+1}(A[T]/T^{\bullet}) \overset{T \mapsto 0}{\longrightarrow} K_{q+1}(A)\right),\]
where \(T\) denotes the \(p\)-typical part 
and \(S\) the symbolic part of Quillen $K$-theory. Bloch's construction was originally
limited to the case $\dim A<p$ and $p\neq 2$. This restriction was removed in \cite{Kato-LCFII}.
\item Following an idea of Deligne, Illusie constructed the de Rham-Witt complex in 
\cite{Illusie} as quotient of \(\Omega^q_{W_n \mathcal{O}_X/W_n(\F_p)}\), such that it is the universal example of a pro-dga with a Verschiebung $V$ satisfying certain properties. 
Then he proves that on this complex an $F$ as above exists.
\item If \(X\) is a smooth scheme over \(k\) with a smooth lift over \(X_n/W_n(k)\), then 
\[W_n \Omega^q_X \cong \mathbf{\mathcal{H}}^q\left(\Omega^*_{X_n/W_n(k)}\right),\]
 and one can show that this isomorphism is independent of the lift, 
 see  \cite[Chap III, (1.5)]{Illusie-Raynaud}. 
 As Illusie--Raynaud explain in  {\em loc. cit.}, it was observed by N. Katz that one can as well take the right hand side of the above isomorphism
 as the definition of the de Rham-Witt sheaves (using local lifts of $X$ over $W_n(k)$ and glue) and that it is possible to construct 
 the structure of a pro-dga  with maps  $F$ and $V$ using this description.
 \item We mention that there are other constructions of the de Rham-Witt complex  by Hesselholt--Madsen \cite{HM}, 
 Langer-Zink \cite{Langer-Zink},  Cuntz--Deninger \cite{Cuntz-Deningeer}, Hesselholt \cite{He-DRW}, 
 and Bhatt-Lurie-Mathew \cite{BLM}, each of which works in a different generality, but they all agree for smooth schemes over a perfect field of positive characteristic.
\end{enumerate}

\end{remark}

\begin{theorem}[{\cite[Chap II, Thm 1.4]{Illusie}}]
Let \(X\) be a smooth scheme over a perfect field \(k\) of characteristic $p>0$ and let
\(u \colon (X/W_n(k))_{\text{crys}} \to X_{\text{Zar}}\) be the change of sites map. Then there is an isomorphism
\[Ru_* \mathcal{O}_{X/W_n(k), \text{crys}} \cong W_n \Omega^\bullet_X.\]
\end{theorem}
We remark that  Bloch proved in \cite[Chap. III, Thm (2.1)]{Bloch} such an isomorphism in the limit over $n$ 
using his $K$-theoretic construction of the de Rham-Witt complex (as pro-object) under the additional assumption, that
$\dim X<p$ and $p\neq 2$. These assumptions were later removed by Kato, see \cite[p. 635, Rem 2]{Kato-LCFII}.

\begin{theorem}[{\cite[Chap I, Thm 4.1 and Chap II, Thm 2.2]{Eke}}]
Let \(X\) be a smooth scheme over \(k\) and let
\[\pi \colon W_n X = (\vert X \vert, W_n \mathcal{O}_X) \to \Spec W_n(k)\]
be the finite type morphism of schemes induced by the structure map of $X$. Then
\[\pi^!W_n(k) \cong W_n \Omega^{\dim X}_X[\dim X],\]
where $\pi^!$ denotes the exceptional inverse image in the derived category of $\mathcal{O}$-modules from Grothendieck duality.
Furthermore, there is a canonical isomorphism
\[W_n \Omega^j_X \overset{\simeq}{\longrightarrow} R \mathcal{H}om_{W_n\mathcal{O}_X}(W_n \Omega^{\dim X-j}_X, W_n \Omega^{\dim X}_X).\]
\end{theorem}

Using the  above theorem, Gros constructed in \cite{Gros} a  pushforward 
\[f_* \colon Rf_* W_n \Omega^j_Y \longrightarrow W_n \Omega^{j-r}_X[-r],\]
for a proper morphism \(f \colon Y \to X\) of relative dimension $r$ between  smooth schemes.

\begin{proposition}\label{prop:deRhamWittReciprocity}
The de Rham--Witt sheaves are reciprocity sheaves, i.e., 
\[W_n \Omega^j \in \RSC_{\Nis}\]
for all \(j \geq 0\). Furthermore, the maps $d$, $R$, $F$, $V$ are compatible with the transfer structure and hence  are morphisms of reciprocity sheaves.
\end{proposition}
\begin{proof}
The finite transfers structure on $W_n\Omega^j$ and its compatibility with $d$, $R$, $V$, $F$ 
is a consequence of  \cite[Theorem 3.4.6]{CR12}. We recall the definition:
For \(Z \in \Cor(X, Y)\) the correspondence action is given by the composition
\begin{align*}
Z^* \colon W_n\Omega^j(Y) &\overset{p^*_Y}{\longrightarrow} W_n \Omega^j(X \times Y) \\
                          &\overset{\cup {\rm cl}_Z}{\longrightarrow} H^{\dim Y}_{Z}(X \times Y, W_n \Omega^{j+\dim Y}) \\
                          &\overset{p_{X*}}{\longrightarrow} W_n \Omega^j(X),
\end{align*}
where ${\rm cl}_Z$ denotes the cycle class and $p_{X*}$ denotes the pushforward with supports from \cite[2.3]{CR12}.
Note that since $Y$ does not need to be proper, this pushforward only exists with support in the finite $X$-scheme $Z$.
We also want to point out the compatibility of the pushfoward, and hence the correspondence action with $d$, $R$, $F$, $V$, is not obvious
and requires Ekedahl's careful analysis of the behaviour of this maps under duality, see \cite[Chap III]{Eke}.
We find \(W_n \Omega^j \in \textbf{NST}.\)
It remains to show that any form \(a \in W_n \Omega^j(X)\) has a modulus. This is similar to the case of Kähler differentials in the proof of \cref{prop:KahlerReciprocity}, see \cite[Thm B.2.2]{BSY} for details.
\end{proof}

Since $\RSC_{\Nis}$ is an abelian category (see Corollary \ref{cor:RSC-ab}),  \Cref{prop:deRhamWittReciprocity} 
gives us many more examples of reciprocity sheaves by taking kernels and quotients of the maps $d$, $R$, $F$, $V$.
In particular we obtain:
\begin{enumerate}
    \item \(W_n \Omega^* \in \text{Comp}^+(\RSC_{\Nis})\) represents the complex of sheaves sending a smooth scheme \(X\) to \(Ru_*\mathcal{O}_{X/W_n, \text{crys}}\).
    \item  We have 
    \begin{equation}\label{eq:Binfty}
        B^\infty W_n \Omega^j := \bigcup_{r\geq 0} F^r d W_{n+r}\Omega^{j-1} \in \RSC_{\Nis}.
    \end{equation}
    \item The \textit{generalized} Artin--Schreier--Witt sequence  is the exact sequence
    \[0 \to W_n \Omega^j_{X, \log} \to W_n \Omega^j_{X}/B_{\infty} \overset{\overline{F}-1}{\longrightarrow} W_n \Omega^j_X/B_{\infty} \to 0\]
    on \(X_{\et}\), where $W_n \Omega^j_{X, \log}$ is the subsheaf of $W_n\Omega^j_X$ which locally is generated by ${\rm dlog}$-forms 
    (see \cite[Chap I, 5.7]{Illusie}), $B^\infty = B^\infty W_n \Omega^j$, 
    and $\overline{F}: W_n \Omega^j_{X}/B_{\infty}\to W_n \Omega^j_{X}/B_{\infty}$ is induced by 
    "lifting to level $n+1$ and apply  $F:W_{n+1}\Omega^j_X\to W_n\Omega^j_X$" (this operation is well-defined on the quotient 
    modulo $B_\infty$). The exactness of this sequence can be deduced from \cite[\S 1, Lem 2]{CSS} and 
    \cite[I, (3.21.1.1), (3.21.1.3) and Prop 3.26 and its proof]{Illusie}.
    By a famous theorem of Geisser-Levine \cite{GL} we have 
    \[ W_n \Omega^j_{X, \log}[-j]\cong \Z/p^n(j)_{X_{\et}},\]
    where the right hand side denotes the \'etale motivic complex of weight $j$ with $\Z/p^n$-coefficients.
    Let  \(\epsilon \colon Sm_{\et} \to Sm_{\Nis}\) be  the change of sites map. Since $W_n \Omega^j_{X}/B_{\infty}$ is a direct limit of
    sheaves which are successive extensions of coherent $\mathcal{O}$-modules, it is acyclic for $R\epsilon_*$.
    Thus we obtain
    \begin{equation}\label{eq:ASW}
    R\epsilon_*\Z/p^n(j) \cong \left(W_n \Omega^j /B_{\infty} \overset{\overline{F}-1}{\longrightarrow} W_n \Omega^j/B_{\infty}\right)[-j] \in D^b(\RSC_{\Nis}).
    \end{equation}
    \item Moreover, by a result of Voevodsky the prime-to-$p$-part of \(R^i\epsilon_*(\mathbb{Q}/\Z(j))\) is homotopy invariant, combining 
    this with the above  yields
    \[R^i\epsilon_*(\mathbb{Q}/\Z(j))\in \RSC_{\Nis},\quad  \text{for all } i, j.\]
    By the above  this is not homotopy invariant  only for $i=j+1$.
    In particular, the Brauer group defines a reciprocity sheaf:
    \[X \mapsto \text{Br}(X) = H^0(X, R^2\epsilon_*(\mathbb{Q}/\Z(1))) \in \RSC_{\Nis}.\]
\end{enumerate}

\section{Computation of the modulus in examples}
In this section we give some computations of the modulus in certain examples. We will see that the modulus detects higher poles and ramifications, which is not captured by the classical \(\A^1\)-invariant theory. 

We let \(L\) be a henselian discrete valuation field of geometric type over the perfect base field \(k\), i.e.,
\[L = \text{Frac }\mathcal{O}^h_{U, x},\]
where \(U\) is a smooth $k$-scheme and \(x \in U^{(1)}\).

For \(F \in \RSC_{\Nis}\) we let \(F(L) := F(\Spec L)\) and \[\tilde{F}(\mathcal{O}_L, \mathfrak{m}^{-n}):= \tilde{F}(\Spec \mathcal{O}_L, n \cdot \{\text{closed pt}\}).\]
By \cite[Thm 4.15(4)]{RS1} we obtain for any proper modulus pair $(X,D)$ that \(\tilde{F}(X, D)\) equals
\begin{multline}\label{eq:Comp-Mod1}
\left\{ 
a \in F(X \setminus D ) \hspace{2mm}\Bigg\vert
\rho^*a \in  \tilde{F}\left(\mathcal{O}_L, \mathfrak{m}_L^{-v_L(\rho^*D)}\right) \quad
\forall L\quad \forall \rho  \in  (X \setminus  D )(L)
\right\}.
\end{multline}
In order to understand the modulus sheaf \(\tilde{F}\) we have to study the filtration
\[F(\mathcal{O}_L) \subset \tilde{F}(\mathcal{O}_L, \mathfrak{m}_L^{-1}) \subset \hdots \subset \tilde{F}(\mathcal{O}_L, \mathfrak{m}_L^{-n}) \subset \hdots \subset F(L)\]
for all \(L\). For \(\A^1\)-invariant Nisnevich sheaves we have $\tilde{F}(\mathcal{O}_L, \mathfrak{m}_L^{-1})=F(L)$.
For a non-$\A^1$-invariant reciprocity sheaf this is an exhaustive increasing filtration, which for varying $L$ is infinite, in the sense
that there exists no natural number $n\ge 0$  such that $F(L)$ is equal to $\tilde{F}(\mathcal{O}_L,\mathfrak{m}^{-n})$, for all $L$, 
see \cite[Lem 5.2]{RSY}.

\begin{definition}\label{defn:level}
The reciprocity sheaf $F$ has {\em level $n\ge 0$}, 
if for any smooth $k$-scheme $X$ and any $a\in F(\A^1\times X)$ the following implication holds:
\[a_{\A^1_z}\in F(z)\subset F(\A^1_z), \quad \text{for all } z\in X_{(\le n-1)}  \Longrightarrow a\in F(X)\subset F(\A^1\times X),\]
where $a_{\A^1_z}$ denotes the restriction of $a$ to $\A^1_z$ and 
$X_{(\le n-1)}$ denotes the set of points in $X$ whose closure has dimension $\le n-1$\footnote{This is equivalent to the motivic conductor of $F$ having level $n$ in the language of \cite{RS1}.}.
\end{definition}
Clearly, $\A^1$-invariant sheaves have level 0.
Any commutative algebraic group \(G\) over \(k\) has level 1 by \cite[Theorem 5.2]{RS1}.
If $F$ has level $n$ it suffices to consider in \eqref{eq:Comp-Mod1} those $L$ which have transcendence degree $\le n$ over $k$.
For example, if the level is $n=1$, this can be interpreted as a {\em cut-by-curves-criterion} for determining the modulus of an element
$a\in F(U)$. If the level is $n=2$ we have a {\em cut-by-surfaces-criterion } etc.

\subsection{Differential forms and rank 1 connections}

\begin{theorem}[{\cite[Chap. 6]{RS1}},{\cite[Chap. 6]{RS-AS}}]\label{thm:Omega}
Let \(\text{char } k = p\ge 0\) and $j\ge 1$.
The modulus sheaf \(\tilde{\Omega^j_{/\Z}}\) has level \(j+1\) and
\[\tilde{\Omega^j_{/\Z}}(\mathcal{O}_L, \mathfrak{m}_L^{-n}) = 
\begin{cases}
\frac{1}{t^{n-1}} \cdot \Omega^j_{\mathcal{O}_L/\Z}(\log t) & \text{if } p=0 \text{ or } (n,p)=1\\
\frac{1}{t^{n}} \cdot \Omega^j_{\mathcal{O}_L/\Z} &\text{if } p>0 \text{ and } p|n,
\end{cases}
\]
where \(t\in\mathfrak{m}_L\) is a local parameter.

Moreover, if $p=0$ and \(\text{Conn}^1(X)\) (resp. \(\text{Conn}^1_{\text{int}}(X)\)) 
denotes the group of isomorphism classes of  (resp. integrable) rank 1 connections on \(X/k\), 
we have
\begin{itemize}
    \item \(\text{Conn}^1\in \RSC_{\Nis}\), has level 2 and  \(\text{Conn}^1_{\text{int}}(X) \in \RSC_{\Nis}\) has level 1.
    \item  \(\tilde{\text{Conn}^1_{\text{int}}}(X, D)\) is the isomorphism classes of integrable rank 1 connections on \(U=X\setminus D\) whose non-log irregularity
    \footnote{The non-log irregularity of a rank one connection $E$ on $\Spec L$ 
    is zero if this connection extends to $\Spec \mathcal{O}_L$ and 
    else is equal to the ${\rm irr}(E)+1$, where ${\rm irr}(E)$ denotes the usual irregularity.}
    is bounded by \(D\).
    \item $h^0_{\A^1}(\text{Conn}^1_{\text{int}})(U)$ is the regular singular rank 1 connections on $U$, where $h^0_{\A^1}(F)$ denotes the 
    maximal $\A^1$-invariant subsheaf of $F$. 
\end{itemize}
\end{theorem}
Using the above formulas and the birational invariance of $\widetilde{\Omega^j_{/\Z}}$ 
(see \eqref{CorKMSY} with $i=0$) it is shown in  \cite[Cor 7.3]{RS-AS} 
(see also \cite[Thm 7.1]{RS-ZNP}) that 
an integral normal Cohen-Macaulay scheme $Y$ of dimension $d$ and of finite type over $k$ has 
{\em pseudo-rational singularities} if for each effective Cartier divisor $R$, such that 
$Y\setminus R$ is smooth, the sheaf $\widetilde{\Omega^d}_{(Y,R)}$ is $S2$, i.e., is completely
determined by its stalks at the zero- and one-codimensional points of $Y$.
Note that in view of \eqref{eq:Comp-Mod1}, this condition can be rephrased as a condition 
on the local filtrations $\widetilde{\Omega^d}(\mathcal{O}_L, \mathfrak{m}^{-n})$
for various $L$ mapping to $Y$.

\subsection{Witt vectors and characters of the abelianized fundamental group}
Let \({\text{char}(k)}=p >0\).
In order to define the Albanese with modulus in higher dimension, Kato--Russel
defined in \cite{KR} the filtration
\[\fil^F_r W_n(L) := \sum_{j \geq 0} F^j\left( \fil^{\log}_{r-1}W_n(L) + V^{n-s}(\fil^{\log}_{r}W_s(L))\right),\quad r\ge 0,\]
where \(\fil^{\log}_r W_n(L) = \{(a_0, \hdots, a_{n-1}) \mid p^{n-1-i} v_L(a_i)\geq -r, \text{ for all } i\}\), and \(s = \min\{n, \text{ord}_p(r)\}\). 

\begin{theorem}[{\cite[Theorem 7.20]{RS1}}]\label{thm:Wn}
The Witt sheaf \(W_n\) has level 1, and 
\[\tilde{W_n}(\mathcal{O}_L, \mathfrak{m}_L^{-r}) = \fil^F_r W_n(L).\]
In particular, 
\[\tilde{\G_a}(\mathcal{O}_L, \mathfrak{m}_L^{-r}) = \left\{
\begin{matrix*}[l]
\mathcal{O}_L, & r \leq 1, \\
\sum_j F^j\left(\frac{1}{t^{r-1}} \mathcal{O}_L\right), & (p,r)=1,\\
\sum_j F^j\left(\frac{1}{t^r} \mathcal{O}_L \right), & p\mid r.
\end{matrix*}
\right. \]
\end{theorem}

Let \(H^1_{\et}(L, \Q/\Z) = \Hom_{\text{cts}}(G_L, \Q/\Z),\) where \(G_L\) denotes the absolute Galois group of \(L\).
As a variant of the  Brylinski--Kato  filtration (\cite{Bry}, \cite{Kat}) Matsuda introduced in \cite{Matsuda} the 
following filtration on \(H^1_{\et}(L, \Q/\Z)\)  
\[\fil_r H^1_{\et}(L, \Q/\Z) := \bigoplus_{l \neq p} H^i_{\et}(L, \Q_l/\Z_l) \oplus \bigcup_n \text{im}\left(\fil^F_r W_n(L) \to H^1_{\et}(L, \Q/\Z)\right),\]
where the maps $\fil^F_r W_n(L) \to H^1_{\et}(L, \Q/\Z)$ are induced by the isomorphism
$H^1_{\et}(L, \Q/\Z)=\varinjlim_V W_n(L)/ (F-1)W_n(L)$ stemming from Artin-Schreier-Witt sequence\footnote{Matsuda does not consider the $F$-saturated filtration, but note that the images of $\fil^F$ and $\fil$ in the quotient $W_n(L)/(F-1)$ 
coincide.}.
This filtration was originally introduced to generalize the Artin-conductor to the case of imperfect residue fields. 

\begin{theorem}[{\cite[Theorem 8.10]{RS1}}]
Let \(\epsilon \colon {\rm Sm}_{\et} \to {\rm Sm}_{\Nis}\) denote the change of sites. Then \(R^1 \epsilon_* \Q/\Z \in \RSC_{\Nis}\) has level 1, and 
\[\tilde{R^1\epsilon_* (\Q/\Z)}(\mathcal{O}_L, \mathfrak{m}_L^{-r}) = \fil_r H^1_{\et}(L, \Q/\Z).\]
\end{theorem}

\begin{remark}
\begin{enumerate}
    \item By work of Abbes--T. Saito  \cite{AbbesSaito09} and Yatagawa \cite{Yatagawa}, we have 
\[\fil_r H^1_{\et}(L, \Q/Z) = \Hom_{\text{cts}}(G_L / G_L^{r+}, \Q/\Z),\]
where \(\{G_L^j\}_{j \in \Q_{\geq 0}}\) is the Abbes--T. Saito 
ramification filtration of \(G_L\) and \(G_L^{r+} = \overline{\bigcup_{s>r} G_L^s}.\) 
\item Similarly as above one can use \eqref{eq:ASW} to determine  $\widetilde{R^{j+1}\epsilon_*(\Q/\Z(j))}$ for $j\ge 1$.
       This is work in progress.
\end{enumerate}
\end{remark}

\subsection{Torsors under finite group schemes in positive characteristic}
Let \(\text{char}(k)=p>0\), and let \(G\) be a finite commutative \(k\)-group scheme. We can write \[G = G_{\text{em}} \times G_{\text{eu}} \times G_{\text{im}} \times G_{\text{iu}},\] where \(G_{\text{em}}\) is an étale multiplicative group (e.g. \(\Z/l\)), \(G_{\text{eu}}\) is an étale unipotent group (e.g. \(\Z/p\)), \(G_{\text{im}}\) is an infinitesimal multiplicative group (e.g. \(\mu_p\)), and \(G_{\text{iu}}\) is an infinitesimal unipotent group (e.g. \(\alpha_p\)).
Consider the presheaf on ${\rm Sm}$
\[X \mapsto H^1(G)(X):= H^1_{\text{fppf}}(X, G)\]
which classifies isomorphism classes of fppf-\(G\)-torsors over \(X\).

\begin{theorem}[{\cite[Theorem 9.12]{RS1}}]
\begin{itemize}
    \item The presheaf \(H^1(G)\) belongs to \(\RSC_{\Nis}\) and has level 2, except for the case when \(G_{\text{iu}}=0\), in which case it has level 1.
    \item $H^1(G_{\text{em}}\times G_{\text{im}}) \in \HI_{\Nis}$.
    \item The map $L\to H^1(\alpha_p)(L)$ induced by 
    the exact sequence of fppf-sheaves $0\to\alpha_p\to \G_a\xrightarrow{F}\G_a\to 0$ restricts to 
      surjections
\[\tilde{\G_a}(\mathcal{O}_L, \mathfrak{m}_L^{-r})\twoheadrightarrow \tilde{H^1(\alpha_p)}(\mathcal{O}_L, \mathfrak{m}_L^{-r}), \quad r\ge 0.\]
\end{itemize}

\end{theorem}

\section{Tensor products and twists}
\subsection{The Lax monoidal structure on \(\RSC_{\Nis}\)}
\begin{definition}
For two reciprocity sheaves \(F\) and \(G\) we define a new reciprocity sheaf by
\[(F, G)_{\RSC_{\Nis}}:= \underline{\omega}_!(h_0^{\boks}(\tilde{F} \otimes_{\MPST} \tilde{G}))_{\Nis} \in \RSC_{\Nis}.\]
\end{definition}
It is not clear that this induces a monoidal structure since it is not clear that this construction is associative. 
Also, it is not clear that this is right exact since \(\underline{\omega}^{\CI}\) is not right exact. 
However, it induces a lax monoidal structure, see \cite[Cor 4.18]{RSY}.

\begin{theorem}[\cite{RSY}]\label{thm:tensor}
\hfill
\begin{enumerate}[label=\arabic*.]
\item Let $\HI_{\Nis}$ denote the category of $\A^1$-invariant Nisnevich sheaves with transfers.
By Voevodsky it has a symmetric monoidal structure denoted by $\otimes_{\HI_{\Nis}}$.
For \(F, G \in \HI_{\Nis}\) we have
\[(F, G)_{\RSC_{\Nis}} = F \otimes_{\HI_{\Nis}} G.\]
\item If \(\text{char } k=0\), then 
\[(G, A)_{\RSC_{\Nis}}=0,\]
for any  commutative unipotent group \(G\) and  abelian variety  \(A\).
 \item\label{thm:tensor3} If \(\text{char }k=0\), then there are isomorphisms
 \[(\G_a, \G_m)_{\RSC_{\Nis}}\xrightarrow{\simeq} \Omega^1_{/\Z},\quad 
 \gamma\otimes a\otimes u\mapsto \gamma^*(p^*a \wedge d\log q^*u),\]
 and if we denote by $I_{\Delta_X}\subset\mathcal{O}_{X\times_{\Z} X}$ the ideal sheaf of the diagonal,
 then
 \[(\G_a, \G_a)_{\RSC_{\Nis}}(X) \xrightarrow{\simeq} 
 H^0(X, \mathcal{O}_{X \times_{\Z} X}/I^2_{\Delta_X}), \quad 
 \gamma\otimes a\otimes b\mapsto \gamma^*(p^*a \otimes q^*b)
 \footnote{We use the transfers structure  stemming from the decomposition 
 $\mathcal{O}_{X \times_{\Z} X}/I^2_{\Delta_X}=\mathcal{O}_X\oplus\Omega^1_{/\Z}$.},\]
 where $\gamma\in \Cor(X, Y\times Z)$, $a\in \G_a(Y)$, $b\in \G_a(Z)$, $u\in \G_m(Z)$,
 and $p$, $q$ denote the projections from $Y\times Z$ to $Y$, $Z$, respectively, 
 and the element $\gamma\otimes a\otimes u$ denotes the images of the corresponding element in 
 $(\G_a\otimes_{\PST} \G_m)(X)$ under the natural map 
 $\G_a\otimes_{\PST} \G_m\to (\G_a, \G_m)_{\RSC_{\Nis}}$, similarly with $\gamma\otimes a\otimes b$.
 \end{enumerate}
\end{theorem}

\begin{definition}
For \(G\) in \(\CI^{\tau, sp}_{\Nis}\) we define
\[G(n) := h_0^{\boks}(G \otimes_{\MPST} \tilde{K^M_n})^{sp}_{\Nis} \in \CI^{\tau, sp}_{\Nis}\]
and
\[\gamma^n G:= \underline{\Hom}_{\MPST}(\tilde{K}^M_n, G) \in \CI^{\tau, sp}_{\Nis},\]
where $K^M_n$ denotes the (improved) Milnor $K$-sheaf (see \cite{Kerz}) and
the upper $sp$ in the first formula denotes the semi-purification functor 
$\CI^\tau_{\Nis}\to \CI^{\tau, sp}_{\Nis}$, which is given by 
$G^{sp}={\rm Im}(G\to \underline{\omega}^*\underline{\omega}_!(G))$.

For \(F\) in \(\RSC_{\Nis}\) we define
\[F\langle 1 \rangle := (F, \G_m)_{\RSC_{\Nis}},\]
and recursively
\[F\langle n \rangle := (F\langle n-1 \rangle)\langle 1 \rangle \in \RSC_{\Nis}.\]
We also define
\[\gamma^n F := \underline{\Hom}_{\PST}(K^M_n, F) \in \RSC_{\Nis}.\]
\end{definition}

Generalizing part of Voevodsky's cancellation theorem (\cite{Cancellation}) Merici and Saito show the following:
\begin{theorem}[{\cite[Corollary 3.6]{MS}}]\label{thm:MSCancellation}
For \(F\) in \(\RSC_{\Nis}\) we have 
\[\gamma^n(\tilde{F}(n)) \cong \tilde{F}\]
and \[\gamma^n(F\langle n \rangle) = F.\]
\end{theorem}

\begin{proposition}[{\cite{RSY}},  {\cite{BRS}}]
There are identities
\begin{itemize}
    \item \(\tilde{\Z}(n)= \underline{\omega}^*K^M_n \in \CI^{\tau, sp}_{\Nis}\)
    \item \(\Z\langle n\rangle = K^M_n \in \RSC_{\Nis}\) 
    \item In char \(k=0\):
    \begin{itemize}
        \item \(\tilde{\G_a}(n) = \tilde{\Omega}^n_{/\Z} \in \CI^{\tau, sp}_{\Nis}\)
        \item \(\G_a\langle n \rangle = \Omega^n_{/\Z} \in \RSC_{\Nis}\)
    \end{itemize}
\end{itemize}
\end{proposition}
The proof of the latter identities uses the computation of \(\tilde{\Omega^n_{/\Z}}(\mathcal{O}_L, \mathfrak{m}_L^{-r})\).

\begin{proposition}
Assume char$(k)=p\ge 7$. Then (see \eqref{eq:Binfty} for notation)
\[(\G_a, K^M_n)_{\RSC_{\Nis}}= \Omega^n/B_\infty\Omega^n.\]
\end{proposition}
\begin{proof}
Using the fact that we have a natural morphism 
$\widetilde{\Omega^n}\to \widetilde{\Omega^n/B_\infty}$ in $\CI^{\tau, sp}_{\Nis},$
and that the inverse Cartier isomorphism induces an endomorphism $\overline{F}:\widetilde{\Omega^n/B_\infty}\to \widetilde{\Omega^n/B_\infty}$
(cf. above \eqref{eq:ASW}), we can use Theorem \ref{thm:Omega} and Theorem \ref{thm:Wn}
to construct a  map $(\G_a, K^M_n)_{\RSC_{\Nis}}\to \Omega^n/B_\infty$ in $\RSC_{\Nis}$ similar 
as in the proof of \cite[Thm 5.20]{RSY}. By Corollary \ref{cor:RSC-ab} and Theorem \ref{thm:ReciprocityPurity}
it suffices to show that it is an isomorphism on any function field $K$.
By \cite[Prop 5.18]{RSY} there is a surjective map  $\Omega^n_K\to (\G_a, K^M_n)_{\RSC_{\Nis}}(K)$ (here we use $p\ge 7$).
The same proof as in \cite[Cor 5.4.12]{IR} shows that this map factors over the quotient $\Omega^n_K/B_\infty$.
By construction of the maps the composition 
\[\Omega^n_K/B_\infty\to (\G_a, K^M_n)_{\RSC_{\Nis}}(K)\to \Omega^n_K/B_\infty\]
is the identity, which completes the proof.
\end{proof}

\begin{remark}
Similarly, one can also show  $\G_a\langle n\rangle= \Omega^n/B_\infty$ (at least for $p\ge 7$). 
This is not immediate from the above. In the induction step $\Omega^n/B_\infty\langle 1\rangle=\Omega^{n+1}/B_\infty$
a description of $\widetilde{\Omega^n/B_\infty}(\mathcal{O}_L,\mathfrak{m}_L^{-n})$ is required. 
The latter group was computed by the second author and will appear somewhere else.
\end{remark}

 \begin{proposition}[{\cite[Thm 11.8]{BRS}}]
Assume $p>0$. There is a natural isomorphism 
\[W_r \Omega^{q-n}\xrightarrow{\simeq} \gamma^n(W_r \Omega^q),\]
which sends a Witt-differential form $\omega\in W_r \Omega^{q-n}(X)$
to the map 
\[\varphi_{\omega}\in \gamma^n(W_r \Omega^q)(X)
=\Hom_{\PST}(K^M_n,\underline{\Hom}_{\PST}(\Z_{\rm tr}(X), W_r \Omega^q)),\] 
which on $Y$ is given by
\[\varphi_{\omega}(Y):
 K^M_n(Y)\to W_r \Omega^q(X\times Y),\quad  a  \mapsto   p_X^*\omega\cdot d\log(p_Y^*a),\]
where $p_X$, $p_Y$ denote the two projections from $X\times Y$ to $X$, $Y$, respectively.
 \end{proposition}
As $W_r\Omega^q$ is a successive extension of certain subquotients of $\Omega^q$ 
this is a consequence of the well known equality  $R^1\pi_*\Omega^q_{X\times \P^1}= \Omega^{q-1}_X$ 
and the fact that a reciprocity sheaf $F$ satisfies 
\[R^1\pi_*F_{X \times \P^1} = (\gamma^1F)_X,\]
where \(\pi \colon \P^1_X \to X\). This is a consequence of the cube-invariance of the cohomology of $F$, see \cite[Thm 9.3]{Sai}.

\section{Cohomology of reciprocity sheaves}\label{sec:CohRec}
In this section we explain some structural results about the cohomology of reciprocity sheaves such as a projective bundle formula, a blow-up formula, a Gysin sequence, the existence of a proper pushforward and the existence of an action of Chow correspondences. 
This has consequences outside the theory of reciprocity sheaves. For example, we obtain new birational invariants of smooth projective varieties 
and obstructions to the existence of zero-cycles of degree one. We survey these applications at the end of this section.

\subsection{Structural results}

\begin{theorem}[Blow-up formula, {\cite[Cor 7.3]{BRS}}]
Let \(G \in \CI^{\tau, sp}_{\Nis}\) and \(\X=(X, D) \in \MCor_{ls}\) 
(see Definition \ref{defn:transversal} for notation). Assume that \(i \colon Z \hookrightarrow X\) is a closed immersion of codimension \(c\) that is transversal to \(D\) (see \ref{defn:transversal}). Furthermore, let \(\rho \colon \tilde{X} \to X\) denote the blow-up of \(X\) in \(Z\), and let  \(\tilde{\X} := (\tilde{X}, D_{\vert \tilde{X}}),\) and \(\mathcal{Z} := (Z, D_{\vert Z})\). Then
\[R\rho_*G_{\tilde{\X}} \cong G_{\X} \oplus \bigoplus_{i=1}^{c-1}i_* \gamma^i G_{\mathcal{Z}}[-i].\]
\end{theorem}

\begin{theorem}[Projective bundle formula, {\cite[Thm 6.3]{BRS}}]\label{thm:pbf}
Let \(G \in \CI^{\tau, sp}_{\Nis}\) and \(\X=(X, D) \in \MCor_{ls}\). Assume that \(\pi \colon P \to X\) is a projective bundle of rank \(n\), and let \(\mathcal{P} := (P, D_{\vert P})\). Then 
\[R\pi_*G_{\mathcal{P}} \cong \bigoplus_{i=0}^n (\gamma^iG)_{\X}[-i].\]
\end{theorem}

The proofs of these two theorems are  intertwined and  go beyond the scope of these lectures.
We just mention some points to compare with classical arguments in the $\A^1$-invariant case:

First one proves that there is a blow-up distinguished triangle (we comment on the proof below).
The projective bundle formula is then proven by induction starting from the $\P^1$-invariance of the cohomology, which is a consequence of the 
cube-invariance proven in \cite[Thm 9.3]{Sai}, and  using  the blow-up triangle in the induction step.
Using the projective bundle formula one can then, similar as Voevodsky, construct a splitting of the blow-up triangle,
where one uses $\boks$-invariance instead of $\A^1$-invariance. 

The essential point in the  proof of the blow-up triangle is to show  the vanishing
\[R^i\rho_* G_{(Y, \rho^*L)}=0,\quad i\ge 1,\]
where $\rho:Y\to \A^2$ is the blow-up in the origin $0$ and $L\subset \A^2$ is a line containing $0$.
Denote by $\pi: Y\to \P^1$ the projection to the exceptional divisor of the blow-up. Then it is not hard to see, that the above vanishing 
is implied by the vanishing
\[H^1(\P^1, \pi_*G_{(Y,\rho^*L)})=0.\]
This is shown in \cite[Lem 2.13]{BRS}. The proof is  a bit technical. However we want to point out that in the course of this proof
one is confronted with certain modulus related problems which do not come up in the $\A^1$-invariant story.
This is why a crucial ingredient in the proof is the following modulus-descent result:
Consider the morphism $\psi_0:\A^1_y\times \A^1_s\to \A^1_x\times \A^1_s$ given
by the $k[s]$-algebra morphism $k[x,s]\to k[y,s]$, $x\mapsto ys$.
It induces a map
\[\psi: \boks^{(1)}_y\otimes \boks^{(2)}_s\to \boks^{(1)}_x\otimes \boks^{(1)}_s \quad \text{in } \MCor,\]
where \(\boks^{(n)}= (\P^1, n \cdot \{0\} + n \cdot \{\infty\})\).
Indeed to check this denote by $\Gamma\subset \P^1_y\times \P^1_{s}\times \P^1_x$
the closure of the graph of $\psi_0$ (as a morphism over $\A^1_s$).
Then the claim holds by the  following identities of divisors on $\Gamma$
\[2 \cdot \{0_{s}\}+ \{0_y\}= \{0_x\}+ \{0_{s}\}, \quad 
  2\cdot \{\infty_{s}\}+ \{\infty_y\}= \{\infty_x\}+ \{\infty_s\},\]
  \[\{0_y\}=\{0_x\}+\{\infty_s\},\quad \{\infty_y\}=\{\infty_x\}+\{0_s\}.\]
In particular, $\psi_0$ does {\em not} define a modulus correspondence 
from $\boks^{(1)}_y\otimes \boks^{(1)}_s$ to $\boks^{(1)}_x\otimes \boks^{(1)}_s$.
\begin{proposition}[{\cite[Prop 2.5]{BRS}}]
Let $G\in \CI^{\tau, sp}_{\Nis}$. With the notation from above 
 $\psi^*$ factors for $\X\in \MCor_{ls}$ as follows
\begin{center}
    \begin{tikzcd}
    & G(\boks^{(1)}_y\otimes \boks^{(1)}_s\otimes \X)\arrow[d]\\
    G(\boks^{(1)}_x\otimes \boks^{(1)}_s\otimes\X)\arrow[r, "\psi^*"] \arrow[ru, dotted] &  G(\boks^{(1)}_y\otimes \boks^{(2)}_s\otimes\X),
    \end{tikzcd}
\end{center}
where the vertical map is induced by the natural morphism $\boks^{(2)}_s\to \boks^{(1)}_s$
and it is injective by the semipurity of $G$.
\end{proposition}
Let us illustrate the above proposition in  the example $G=\widetilde{\Omega^2}_{/k}$ and $\X=(\Spec k, \emptyset)$. 
In this case  $G(\boks^{(1)}_x\otimes \boks^{(1)}_s)= k\cdot {\rm dlog}(x){\rm dlog}(s)$ and we have 
$\psi^*({\rm dlog}(x){\rm dlog}(s))= {\rm dlog}(ys){\rm dlog}(s)={\rm dlog}(y){\rm dlog}(s)\in G(\boks^{(1)}_y\otimes \boks^{(1)}_s)$.

\begin{example}
We spell-out the projective bundle formula in two concrete cases (the blow-up formula is similar). Let the situation be as in \Cref{thm:pbf}.
\begin{itemize}
    \item In \(\text{char}(k)=0\) we have
    \[R\pi_*\Omega^j_{P/\Z}(\log D_{\vert P})(D_{\vert P} -  D_{\vert P, {\rm red}} ) = 
    \bigoplus_{i=0}^n \Omega^{j-i}_{X/\Z}(\log D) (D-  D_{\rm red})[-i].\]
    Taking \(D=\emptyset\) recovers the classical projection bundle formula.
    \item In positive characteristic $p$ with \(D=\emptyset\), we have
    \[R\pi_* \left( R^{j+1} \epsilon_* \Z/p^r(j)\right)_P = \bigoplus_{i=0}^n \left( R^{j-i+1}\epsilon_* \Z/p^r(j-i)\right)_X[-i].\]
    Note that this can also be deduced from the projective bundle formula for the Hodge-Witt cohomology  by Gros \cite[I, Thm 4.1.11]{Gros}.
\end{itemize}
We remark, that in the second example we have to take the empty divisor since 
the formula $\gamma^i (R^{j+1} \epsilon_* \Z/p^r(j))_{(X,D)}=(R^{j-i+1} \epsilon_* \Z/p^r(j-i))_{(X,D)}$ is only known for $D=\emptyset$,
in which case it follows from the exactness of $\gamma^i$ and \cite[Thm 11.8]{BRS}. In characteristic zero we have 
$\gamma^i\widetilde{\Omega^j_{/\Z}}=\widetilde{\Omega^{j-i}_{/\Z}}$ by \cite[Cor 11.2]{BRS}.
\end{example}


Similar to Voevodsky's (\cite[Prop. 3.5.4]{TriCa}) we obtain a Gysin triangle.
\begin{theorem}[Gysin sequence {\cite[Thm 7.16]{BRS}}]\label{thm:gysin}
Let \(G \in \CI^{\tau, sp}_{\Nis}\) and \(\X=(X, D) \in \MPST_{ls}\). Assume that \(i \colon Z \hookrightarrow X\) is a closed immersion 
of codimension \(c\) that is transversal to \(D\), in the sense of \ref{defn:transversal}. Furthermore, let \(\mathcal{Z} := (Z, D_{\vert Z})\), and let \(\rho \colon \tilde{X} \to X\) denote the blow-up of \(Z\) in \(X\) with \(E = \rho^{-1}(Z)\) the exceptional divisor. Then there is an exact triangle
\[i_* \gamma^c G_{\mathcal{Z}}[-c] \overset{g_{\mathcal{Z}/\X}}{\to} G_{\X} \to R\rho_*G_{(\tilde{X}, D_{\vert \tilde{X}+E})} \overset{\partial}{\to} i_* \gamma^c G_{\mathcal{Z}}[-c+1]\]
in \(D(X_{\Nis})\).
\end{theorem}


\begin{example}
In \(\text{char}(k)=0\) and \(c=1\) we get an exact sequence
        \[\begin{matrix*}[l]
            0 \to  \tilde{\text{Conn}}^1(X, D) \to  \tilde{\text{Conn}}^1(X, D+Z)  \to  H^0(Z, \mathcal{O}_Z(i^* D - (i^*D)_{\rm red} ))/\Z \vspace{5mm}\\
     
            \overset{g_{\mathcal{Z}/\X}}{\to} H^1\left(X, \frac{\Omega^1_{X/k} (\log D) (D - D_{\rm red})}{d\log(j_* \mathcal{O}^\times_{X\setminus D })} \right)  \to  H^1\left(X, \frac{\Omega^1_{X/k} (\log D+Z) (D - D_{\rm red} )}{d\log(j_* \mathcal{O}^\times_{X\setminus (D+Z)})} \right) 
        \end{matrix*}\]
    and for \(c \geq 2\) we have
    \[\tilde{\text{Conn}}^1(X, D) \cong \tilde{\text{Conn}}^1(\tilde{X}, \rho^* D+E),\]
    where $\text{Conn}^1(X)$ denotes the  isomorphism classes of rank 1  connections on $X$.
\end{example}
\begin{example}
    Let \(\text{Lisse}^1 \in \RSC_{\Nis}\) be the sheaf whose sections over \(X\) are the lisse \(\overline{\Q}_l\)-sheaves of rank 1, and \(\tilde{\text{Lisse}}^1(X, D)\) the lisse \(\overline{\Q}_l\)-sheaves of rank 1 on \(X \setminus D \) with Artin conductor less than or equal to \( D\). Then for \(\text{char}(k)=p>0\), \(l \neq p\), and \(c \geq 2\), we have
    \[\tilde{\text{Lisse}}^1(X, D) = \tilde{\text{Lisse}}^1(\tilde{X}, \rho^*D+E).\]
\end{example}

We can now define a proper correspondence action on reciprocity sheaves.
\begin{definition}
Let \(S\) be a scheme of finite type over \(k\) and let \(C_S\) denote the category of \textit{proper (Chow) correspondences}, i.e., its objects are \(S\)-schemes \(X \overset{f}{\to} S\) such that \(X\) is quasi-projective, and smooth over \(k\), and \(f\) is a morphism of finite type. Morphisms in \(C_S\) between (connected) objects \(X\) and \(Y\) are elements in
\[C_S(X, Y) := \varinjlim_{Z \subset X \times_S Y}CH_{\dim X}(Z),\]
where the direct limit runs over all closed subschemes $Z\subset X \times_S Y$ 
which are proper over  $X$.
The composition of morphisms in this category is defined  using Fulton's refined intersections.
\end{definition}

\begin{definition}\label{defn:proper-cor-action}
For \(F \in \RSC_{\Nis}\), and objects \(X \overset{f}{\to} S\) and \(Y \overset{g}{\to} S\) in \(C_S\), together with an element 
\(\alpha \in C_S(X, Y)\), we define the \textit{proper correspondence action} 
\[\alpha^* \colon Rg_*F_Y \to Rf_*F_X\qquad \text{in } D(S_{\Nis})\]
 by
\begin{enumerate}
    \item\label{defn-cor1} pulling back to \(X \times Y\)
    \item\label{defn-cor2} cupping with \(\alpha\) (minding the support)
    \item\label{defn-cor3} pushing forward to \(X\) (using the properness of the support over \(X\)).
\end{enumerate}

For \ref{defn-cor2} we note that the Gersten resolution for Milnor $K$-theory 
yields an identification $CH_{\dim X}(V)=H^e_V(X \times Y, K^M_e)$ (Bloch formula with support),
where \(V \subset X \times_S Y\) is proper over \(X\) and $e=\dim Y$, and hence 
\(\alpha \in CH_{\dim X}(V)\) corresponds to a map \(\colon \Z[-e] \to R \underline{\Gamma}_V(K^M_e)\). The cupping  with \(\alpha\)  is then defined as  the composition
\begin{multline*}
\gamma^e F[-e] \overset{\alpha}{\longrightarrow} \gamma^e F \otimes_{\Z}^L R\underline{\Gamma}_V(K^M_e) \longrightarrow  R\underline{\Gamma}_V(\gamma^eF \otimes_{\MPST} K^M_e) \\
= R\underline{\Gamma}_V\left(\underline{\text{Hom}}_{\MPST}(K^M_e, F) \otimes_{\MPST} K^M_e\right) \overset{\text{adj.}}{\longrightarrow} R\underline{\Gamma}_V F.
\end{multline*}

The construction of the pushforward in \ref{defn-cor3}
follows the classical method, but we have to
keep track of the support, as the projection $X\times Y\to X$ does not need to be projective:
Take a closed embedding of  $Y$ into an open $U$ of a projective space $P$ over $\Spec k$.
Then the pushforward is defined by using the Gysin map with support in $V$ along 
the closed embedding $X\times Y\hookrightarrow X\times U$, 
by excision the cohomology of $X\times U$ with support in $V$ agrees with the  
cohomology of $X\times P$ with support in $V$ and  we can use the projective bundle formula 
to pushforward to $X$. The cancellation theorem  (\cref{thm:MSCancellation}) is used 
to cancel  twists. 
\end{definition} 

We obtain a functor  \(C_S \to D(S_{\Nis})\) given by \((X \overset{f}{\to} S) \mapsto Rf_* F\).

\subsection{Applications}
In this section we survey applications of the results presented in the previous sections, see \cite[Chap 10, 11]{BRS} for more details.

\subsubsection{Obstructions to the existence of zero cycles in degree 1.}
\begin{theorem}[{\cite[Cor 10.2]{BRS}}]
Let \(F \in \RSC_{\Nis}\). Let \(f \colon X \to S\) be a projective dominant map between smooth $k$-schemes.
Assume there exists a degree 1 zero cycle on $X_K$, where $K=k(S)$. Then
\[f^* \colon H^i(S, F_S) \longrightarrow H^i(X, F_X)\]
is split-injective.
\end{theorem}
\begin{proof}
Take \(\overline{\xi}\) as a lift of \(\xi\in CH_0(X_K)^{{\rm deg}= 1}\) under the map
\[C_S(S, X) = CH_{\dim S}(S \times_S X) \to CH_0(X_K).\]
Then $f_*\overline{\xi}=[S]=\text{id}_S\in CH_{\dim S}(S)=C_S(S,S)$. Hence the splitting
\[\overline{\xi}^* \circ f^* = (f_*\overline{\xi})^* = \textit{id} \colon H^i(S, F) \to H^i(X, F) \to H^i(S, F).\]
\end{proof}

\subsubsection{Generalized Brauer--Manin obstruction for zero cycles}

Let \(S\) be a smooth projective curve over $k$ with function field \(K = k(S)\).
Furthermore, let \(f \colon X \to S\) be a projective dominant map with $X$ smooth, 
and choose \(v \in S_{(0)}\). Then for \(\alpha_v \in CH_0(X_{K_v})\) with lift \(\overline{\alpha}_v \in CH_1(X_{S_v})\),
where \(K_v\) and \(S_v\) are the henselization of \(K\) and \(S\) at \(v\), the map
\[\Psi(\alpha_v) := \overline{\alpha}_v^* \colon f_* F_{X_{S_v}} \to F_{S_v}\]
 depends only on $\alpha_v$ and not on the lift $\overline{\alpha}_v$.
Indeed if $\alpha'_v$ is a different lift then $\beta=\overline{\alpha}_v-\alpha'_v$ can be represented 
by a cycle supported in the special fiber $X_v$ and 
it follows from the construction of the proper correspondence action 
in Definition \ref{defn:proper-cor-action}, that $\beta^*: f_* F_{X_{S_v}} \to F_{S_v}$
factors via $\underline{\Gamma}_v(F_{S_v})$ which vanishes by Saito's Purity Theorem, 
see \ref{thm:ReciprocityPurity}. Taking the first cohomology with support in $v$ yields a map
\[H^1_v(S, f_*F_X) \overset{\Psi(\alpha_v)}{\longrightarrow} H^1_v(S, F) \to H^1(S, F).\]
Thus we obtain a map
\[\Psi \colon \prod_{v \in S_{(0)}} CH_0(X_{K_v}) \to \text{Hom}\left( \bigoplus_{v \in S_{(0)}} H^1_v(S, f_* F_X), H^1(S, F) \right).\]
Moreover, there is a map
\[\iota \colon F(X_K) \to \bigoplus_{v \in S_{(0)}} H^0(S_v \setminus \{v\}, f_*F_X) \overset{\partial}{\longrightarrow} \bigoplus_{v \in S_{(0)}} H^1_v(S, f_*F_X).\]

\begin{theorem}[{\cite[Cor 10.4]{BRS}}]
If \(\Psi((\alpha_v)_v) \circ \iota \neq 0\), then there does not exist \(\alpha \in CH_0(X_K)\) such that \(\alpha \mapsto (\alpha_v)_v.\)
\end{theorem}
\begin{proof}
Take \(\alpha \mapsto \alpha_v\) and \(\overline{\alpha} \in CH_1(X)\) a lifting of \(\alpha\). We then get a diagram
\begin{center}
    \begin{tikzcd}[row sep=large]
    F(X_K) \arrow[r, "\iota"] \arrow[rr, "0", bend left=20] & \bigoplus_v H_v^1(S, f_*F_X) \arrow[r] \arrow[d, "\sum_v \Psi(\alpha_v)"] & H^1(X, F) \arrow[dl, "\overline{\alpha}^*"] \\
    & H^1(S, F),
    \end{tikzcd}
\end{center}
where the second map in the horizontal sequence is
the composition of summing the forget-support-maps $H^1_v(S, f_*F_X)\to H^1(S, f_*F_X)$
with the natural map $H^1(S, f_*F_X)\to H^1(X, F)$ induced by 
$f_*F_X\to Rf_*F_X$. 
\end{proof}


\begin{remark}
Assume \(k = \F_q\) is a finite field with $q$ elements and \(F = \text{Br}\).
Using the Cousin resolution \eqref{eq:Cousin} of ${\rm Br}_S$  and
the fact that in the case at hand we have 
\[H^1_v(S,{\rm Br}_S)={\rm Br}(K_v)/{\rm Br}(S_v)={\rm Br}(K_v),\]
the Brauer-Hasse-Noether Theorem in the function field case yields
\[H^1_{\Nis}(S, {\rm Br})= {\rm Coker}({\rm Br}(K)\to \bigoplus_{v\in S_{(0)}} {\rm Br}(K_v))= \Q/\Z,\]
see \cite[XIII, \S3, Thm 2 and \S6, Thm 4]{Weil-BNT}. 
Thus \(\Psi\)  equals the map
\[\prod_{v \in S_{(0)}} CH_0(X_{K_v}) \longrightarrow \text{Hom}\left( \bigoplus_{v \in S_{(0)}} \frac{\text{Br}(X_{K_v})}{\text{Br}(X_{S_v})}, \Q/\Z\right),\]
which is the classical Brauer--Manin obstruction for zero-cycles in the function field case. 
\end{remark}

\subsubsection{Stably birational invariance}
\begin{definition}
Let \(X \overset{f}{\to} S\), \(Y \overset{g}{\to} S\) be objects of \(C_S\) and assume that $X$ and $Y$
are integral. We say that \(f\) and \(g\) are {\em properly birational over} \(S\) 
if there exists proper birational \(S\)-maps \(Z \to X\) and \(Z \to Y\) ($Z$ could be singular).

The maps \(f\) and \(g\) are said to be {\em stably properly birational over} \(S\) if there exist vector bundles \(V\) over \(X\) and \(W\) over \(Y\) such that \(\P(V)\) and \(\P(W)\) are properly birational over \(S\).
\end{definition}

\begin{example}
If \(S\) is singular, and \(X\) and \(Y\) are two different resolutions of \(S\), then they are properly birational over \(S\).

If $f$ is proper and we take \(S = Y = \Spec k\), then  \(f\) and $\text{id}_{\Spec k}$ are stably properly
birational if and only if  $X$ is stably rational over $k$.
\end{example}

\begin{theorem}[{\cite[Thm 10.7]{BRS}}]
Any \(F \in \RSC_{\Nis}\) is a stably properly birational invariant over \(S\), that is, for every stably properly birational \(S\)-schemes \(X \overset{f}{\to} S, Y \overset{g}{\to} S \in C_S\), we have an isomorphism
\[f_* F_X \overset{\simeq}{\longrightarrow} g_* F_Y.\]
\end{theorem}
\begin{proof}
This follows from the projective bundle formula, purity, and the correspondence action.
\end{proof}

\begin{theorem}[{\cite[Thm 10.10]{BRS}}]
Let \(X \overset{f}{\to} S, Y \overset{g}{\to} S \in C_S\) be properly birational over \(S\) and let \(F \in \RSC_{\Nis}\). Assume that $F(K)=0$ for all function fields $K/k$ of transcendence degree $\le d-1$,
where $d=\dim X=\dim Y$. Then
\[R g_*F_Y \overset{\simeq}{\longrightarrow} Rf_*F_X.\]
\end{theorem}
\begin{proof}
Take a closed subscheme \(Z \subset X \times_S Y\) mapping properly and birationally to \(X\) and \(Y\). Then \(Z \circ Z^t = \Delta_Y + \epsilon\) with \(p_{Y*}\epsilon \in CH^{\geq 1}(Y),\)
and the condition on \(F\) implies that \(\epsilon^* =0\) on \(Rg_*F_Y\), see \cite[Prop 9.13]{BRS}.
This implies that \(Z \circ Z^t\) acts as the identity on $Rg_*F_Y$; similarly with $Z^t\circ Z$.
\end{proof}

\begin{remark}
Taking \(Y=S\), and \(g\) the identity, yields the vanishing result \[Rf_*F_X \cong 0.\]
\end{remark}

\begin{example}
Assume that \(\dim X = \dim Y = d\). Then the theorem applies to the following list of sheaves:
\begin{itemize}
    \item \(\Omega^d_{/k}\), \,\(\Omega^d_{/k}/d\log K^M_d\);
    \item if \(\text{char } k =p\neq 0\)
    \begin{itemize}
        \item \(W_n\Omega^d/B_{\infty}\),  \(R^i\epsilon_*(\Z/p^n(d))\),
        \item \(G\langle d \rangle\), for \(G\) a smooth unipotent group,    
        \item \(H^1(G)\langle d\rangle\), for \(G\) a finite \(p\)-group over \(k\),
        \item (If furthermore \(k\) is algebraically closed) \(R^d\epsilon_* \Q/\Z(d)\).
    \end{itemize}
\end{itemize}
\end{example}

\begin{remark}
The case for \(\Omega^d_{/k}\) was known before by \cite[Thm 1]{CR11}.
It was later generalized to regular schemes  in \cite{CR15} and \cite{Kovacs}.

In the first case in positive characteristic we use Geisser-Levine \cite[Thm 8.3]{GL}, in the second and third case we use the 
Bloch--Kato--Gabber  Theorem \cite[Thm 2.1]{BK} and in the last case we use additionally the Milnor-Bloch-Kato conjecture,
proven by Rost-Voevodsky  \cite[Thm 6.16]{V11}, to check that the condition $F(K)=0$ for ${\rm trdeg}(K/k)<d$ 
is satisfied in the cases at hand.

There is a version of the theorem where the vanishing $F(K)=0$ for ${\rm trdeg}(K/k)<d$ is replaced by 
the vanishing $\gamma^1F=0$ (which is for example satisfied if $F$ is any smooth commutative unipotent group),
but this requires at the moment resolution of singularities in dimension \(d-1\).
\end{remark}

\begin{corollary}[{\cite[Cor 11.24]{BRS}}]
Let \(X \overset{f}{\to} S\) and \(Y \overset{g}{\to} S\) be  flat, geometrically integral, and projective morphisms between smooth connected \(k\)-schemes. Assume that the generic fiber has index 1, (implying that the Picard schemes \(\text{Pic}_{X/S}\) and \(\text{Pic}_{Y/S}\) are representable).
If \(X\) and \(Y\) are stably properly birational over \(S\), then
\[\text{Pic}_{X/S}[n] \cong \text{Pic}_{Y/S}[n]\]
on \(S_{\Nis}\) for all \(n\).
\end{corollary}

\begin{remark}
The above result is classical for \(S=\Spec k\), with \(k\) algebraically closed.
\end{remark}

\subsubsection{Decomposition of the diagonal}
\begin{definition}
Let \(K\) be a function field over \(k\) and \(X\) a smooth scheme over \(K\) with \(\dim X = d\). We say that the \textit{diagonal of} \(X\) \textit{decomposes} if 
\begin{equation}\label{eq:DecompDiag}
    [\Delta_X] = p_2^* \xi + (i \times \text{id})_* \beta \in CH^d(X \times_K X),
\end{equation}
where \(\xi \in CH_0(X)\) and \(\beta \in CH_d(Z \times_K X)\) for some closed subscheme 
$i:Z\hookrightarrow X$ with \(\text{codim}(Z, X) \geq 1\). 
\end{definition}

This condition was first considered by Bloch--Srinivas for rational coefficients in \cite{Bloch-Srinivas}.
By \cite[Lem 1.5]{CP} an integral smooth projective $k$-scheme \(X\), which is {\em retract rational} (i.e., there exist  open dense subsets \(U \subset X \) and \(V \subset \P^n_K\) together 
with a map \(V \to U\) which has a section), has the property that its diagonal  decomposes.
Hence implications of \eqref{eq:DecompDiag} on cohomology yield obstructions to \(X\) being retract rational over \(K\).

\begin{theorem}[{\cite[Thm 10.13]{BRS}}]\label{thm:DecompDiag}
Let S be the henselization of a smooth k-scheme in a 1-co\-dimensional
point or a regular connected affine scheme of dimension $\le 1$ and of finite type over a
function field $K$ over $k$. 
Let \(f \colon X \to S\) be a smooth and projective morphism, 
and assume that the diagonal of the generic fiber of \(f\) decomposes. Then 
\[F(S) = F(X),\] 
for any \(F \in \RSC_{\Nis}\).
\end{theorem}

\begin{remark}
In \cite[Problem 1.2]{ABBG} the following problem is posed:

Let \(k\) be algebraically closed with \(\text{char}(k)=p>0\) and \(X\) a smooth and proper scheme over \(k\) with decomposition of the diagonal. Do we then have that 
\[H^0(X, R^i\epsilon_* \Z/p(j)) = 0, \quad \forall i\neq 0?\]

\cref{thm:DecompDiag} gives a positive answer to the problem if \(X\) is projective over \(k\). Indeed, if we take \(S=\Spec k\) and \(F=R^i\epsilon_* \Z/p(j)\) we observe that \(F(k)=0\), 
see \eqref{eq:ASW}. 
\end{remark}

\begin{theorem}[{\cite[Cor 11.22]{BRS}}]
Let $f : X \to S$ be a projective morphism between smooth integral and quasi-projective k-schemes. 
Let $\dim S=e$ and $\dim X = d$. Assume
the diagonal of the generic fiber of $f$ decomposes. Then 
\[f_* \colon Rf_*F^d_X \overset{\cong}{\longrightarrow} F_S^e[e-d]\]
is an isomorphism, where $F^d$ is one of the following sheaves
\begin{itemize}
    \item \(\Omega^d_{/k}\), \(\Omega^d_{/k}/d\log K^M_d\)
    \item (\(k\) is algebraically closed): \(R^d\epsilon_* \Q/\Z(d)\)
    \item (\(\text{char } k=p>0\)):
    \begin{itemize}
        \item \(W_n\Omega^d/B_{\infty}\)
        \item  \(R^i\epsilon_*(\Z/p^n(d))\)
    \item  \(G\langle d \rangle\), for \(G\) a smooth commutative unipotent $k$-group
    \item  \(H^1(G)\langle d\rangle\), for \(G\) a finite \(p\)-group over \(k\).
    \end{itemize}
    \end{itemize}
\end{theorem}

\begin{example}
If \(k\) is algebraically closed and \(X\) is smooth and projective over \(k\) of dimension \(d\) such that the diagonal of \(X\) decomposes, then
\[H^i(X, R^{d+1}\epsilon_*\Z/p^n(d)) = 0, \qquad \text{for all } i.\]
\end{example}

\section{Further results}
In this section we give an overview of some more recent results  obtained in \cite{RS-ZNP}  and \cite{RS-AS}.
This section was not part of the original lecture series given at CAS.

\subsection{Zariski-Nagata Purity}
Let $U$ be a smooth $k$-scheme over a perfect base field $k$ and let $K$ be a function field over $k$. 
For any presheaf with transfers we have a pairing
\begin{equation}\label{eq:pairing1}
    F(U_K)\otimes \Cor_K(\Spec K, U_K)\to F(K),\quad (a,\gamma)\mapsto \gamma^*a,
\end{equation}
where $U_K=U\otimes_k K$ and $\Cor_K$ denotes the finite correspondences on smooth $K$-schemes.
Moreover, for $X$ an integral finite type  $k$-scheme of dimension $d$, 
and $D\subset X$ a closed subscheme with ideal sheaf $I_D$ and open complement $U=X\setminus D$, we have by \cite[Thm 2.5]{KS-GCFT} a surjective map
\begin{equation}\label{eq:pairing2}
    \Cor_K(\Spec K, U_K)\twoheadrightarrow H^d(X_{K,\Nis}, K^M_d(\mathcal{O}_{X_K}, I_{D_K})),
\end{equation}
where $K^M_d(\mathcal{O}_{X_K}, I_{D_K})={\rm Ker}(K^M_d(\mathcal{O}_{X_K})\to K^M_d(\mathcal{O}_{D_K}))$. The map is induced
by the isomorphism $\Z\cong H^d_x(X_{K,\Nis}, K^M_d(\mathcal{O}_{X_K}, I_{D_K}))$, for $x$ a closed point in $U_K$,
stemming from the Gersten resolution, see \cite[Prop 10 (8)]{Kerz}.

\begin{theorem}[{\cite[Thm 1.6]{RS-ZNP}}]\label{thm:ZNP}
Let $X$ be a smooth integral projective $k$-scheme of dimension $d$ 
and let  $\sum_{i=1}^r   D_i$ be an SNCD divisor with complement  $U=X\setminus \cup_{i=1}^r D_i$.
Let $\mathfrak{n}=(n_1,\ldots, n_r)\in (\mathbb{N}_{\ge 1})^r$ and set $D_{\mathfrak{n}}=\sum_{i=1}^r n_i D_i$.
Let $F\in \RSC_{\Nis}$. Then the following is equivalent for $a\in F(U)$:
\begin{enumerate}[label=(\roman*)]
    \item\label{ZNPi} $a\in \widetilde{F}(X,D_{\mathfrak{n}})$;
    \item\label{ZNPii} 
    $a\in \widetilde{F}(\mathcal{O}_{X,\eta_i}, \mathfrak{m}_{\eta_i}^{-n_i})$, all $i$, where $\eta_i$ is the generic point of $D_i$;
    \item\label{ZNPiii} for any function field $K$ the map $(a_K, -)_{U_K/K}\colon \Cor_K(\Spec K, U_K)\to F(K)$ induced by 
    \eqref{eq:pairing1} factors via \eqref{eq:pairing2} to give a map
    \begin{equation}\label{eq:pairing-mod}
        (a_K,-)_{(X_K,D_{\mathfrak{n},K})/K}\colon H^d(X_{K,\Nis}, K^M_d(\mathcal{O}_{X_K}, I_{D_{\mathfrak{n},K}}))\to F(K),
    \end{equation}
    where $a_K$ denotes the pullback of $a$ to $F(U_K)$.
\end{enumerate}
\end{theorem}
Note that the equivalence of \ref{ZNPi} and \ref{ZNPii} is a statement of Zariski-Nagata type (``purity of the branch locus'').
It continues to hold if $X$ is assumed to be quasi-projective and $(X, \sum_i D_i)$ has a projective SNC-compactification
(which is always the case in characteristic 0), see the proof of \cite[Cor 6.10]{RS-ZNP}.
In case  $r=1$ (i.e., $D_{\mathfrak{n}}$ has just one component) the equivalence 
of \ref{ZNPi} and \ref{ZNPii} also follows  from \cite[Cor 8.6(2)]{Sai}.
For $\mathfrak{n}=(1,\ldots, 1)$ (i.e., $D_{\mathfrak{n}}$ is reduced SNCD), it follows from  \cite[Cor 2.4]{Sailog}.

Observe that from the equivalence of \ref{ZNPi} and \ref{ZNPiii} we obtain a map
\begin{equation}\label{eq:rec-map}
CH_0(X|D_{\mathfrak{n}})\twoheadrightarrow H^d(X_{\Nis}, K^M_d(\mathcal{O}_X, I_{D_{\mathfrak{n}}}))\to \Hom (\widetilde{F}(X, D_{\mathfrak{n}}), F(k)),
\end{equation}
where $CH_0(X|D_{\mathfrak{n}})$ denotes the Chow group of zero cycles with modulus introduced in 
\cite{Kerz-Saito} and the first map is induced from \eqref{eq:pairing2}. 
(This factorization is proved in various cases, e.g., by Krishna and Gupta-Krishna.
For the situation at hand, see \cite{RS-cycle-class}). Taking the limit of the composition 
we get a natural map
\[C(U):=\varprojlim_{\mathfrak{n}} CH_0(X|D_{\mathfrak{n}})\to \Hom(F(U), F(k)).\]
If $k$ is a finite field and $F=\Hom_{\rm cts}(\pi_1^{\rm ab}(-), \Q/\Z)$, then 
$F(k)\cong\Q/\Z$ and this map
is the {\em reciprocity homomorphism} constructed in \cite[Prop 3.2]{Kerz-Saito},
similarly the limit over $\mathfrak{n}$ of the second map in \eqref{eq:rec-map}
is the {\em reciprocity homomorphism} constructed in \cite{KS-GCFT}. 

\subsection{Abbes-Saito formula}
Let $F\in \RSC_{\Nis}$. We ask the following:
\begin{enumerate}[label=\arabic*.]
    \item Is it possible to give a more computable description of $\widetilde{F}$, in particular without using the transfers structure of $F$?
    \item\label{Q:2} For $X$ smooth and $D\subset X$ a smooth divisor, we get from the Gysin sequence 
         \ref{thm:gysin} an isomorphism 
         \[\widetilde{F}(X,D)/F(X)\cong \underline{\Hom}_{\PST}(\G_m,  F)(D).\]
         Can we also describe the quotients $\widetilde{F}(X,nD)/\widetilde{F}(X,(n-1)D)$, for $n\ge 2$?
\end{enumerate}
In \cite{RS-AS} it is shown  that these questions can be approached using a method introduced by Abbes and (Takeshi) Saito
in \cite{AbbesSaito11} and \cite{TakeshiSaito}  to study the ramification of Galois torsors by means of
dilatations. For simplicity we assume in the following that:

\begin{enumerate}
    \item[$(*)$] $X$ is smooth, $D$ is a smooth divisor on $X$, $U=X\setminus D$, and $(X,D)$ has a projective SNC compactification, i.e.,
    there exists an open embedding  $X\hookrightarrow\overline{X}$ into a smooth projective scheme $\overline{X}$ such that $\overline{X}\setminus U$ is the support of a divisor with simple normal crossings\footnote{In \cite{RS-AS} more generally the case where $D$ is a SNCD is considered.}.
\end{enumerate}

The dilatation $P^{(nD)}_X$, for $n\ge 1$, is the blow-up of $X\times X$ in $nD$ diagonally embedded and with the strict transforms
of $X\times nD$ and $nD\times X$ removed. It comes with two maps
\[p_1, p_2 \colon P^{(nD)}_X\rightrightarrows X\]
induced by the two projection maps $X\times X\to X$. Note that the open immersion $U\times U\hookrightarrow X\times X$ extends to
an open immersion $U\times U\hookrightarrow P^{(nD)}_X$.

\begin{theorem}[{\cite[Theorem 1]{RS-AS}}]\label{thm:AS}
Assume $(*)$ and let $n\ge 1$. Then
\[\widetilde{F}(X, nD)=\{a\in F(U)\mid p_1^*a=p_2^*a \text{ in } F(U\times U)/F(P^{(nD)}_X)\}.\]
\end{theorem}
In particular the theorem applies to $F=H^1_{\rm fppf}(-,G)$\footnote{In this case even without the assumption on the existence of a projective SNC compactification.}, where $G$ is a commutative finite $k$-group scheme (not necessarily \'etale).
A version of this formula was proved by Abbes-Saito for $G$   any \'etale $k$-group 
(not necessarily commutative). 
For more details  and precise references see \cite[Ex 2.12]{RS-AS}.

The proof of the above theorem uses heavily the theory of higher local symbols along Par\v{s}in chains for reciprocity sheaves
developed in \cite{RS-HLS}, which in turn relies on Theorem \ref{thm:ZNP} and Section \ref{sec:CohRec}.

Using Theorem \ref{thm:AS} we obtain the following partial description of the quotients considered in question \ref{Q:2} above.

\begin{theorem}[{\cite[Thm 4.12]{RS-AS}}]\label{thm:AS2}
Assume $(*)$ and let $n\ge 2$. Then there is an injective map
\[{\rm char}^{(nD)}_F \colon \frac{\widetilde{F}(X,nD)}{\widetilde{F}(X, (n-1)D)}\hookrightarrow 
H^0\left(D, \Omega^1_X(nD)_{|D}\otimes_{\mathcal{O}_D} \underline{\Hom}_{{\rm Sh}_D}(\mathcal{O}_D, F_D)\right),\]
where $\underline{\Hom}_{{\rm Sh}_D}$ denotes the  internal hom in the category of Nisnevich sheaves of abelian groups on smooth
schemes over $D$.
\end{theorem}
Some comments:
\begin{itemize}
\item If ${\rm char}(k)=p>0$, it follows that  the quotient on the left hand side is $p$-torsion.
This can be seen as a analogue of \cite[Cor 2.28]{TakeshiSaito} for reciprocity sheaves. 
\item The characteristic form  for 
\[F= H^1_{\et}(-, \Q/\Z)=\Hom_{\rm cont}(\pi_1^{\rm ab}(-), \Q/\Z)\]
factors via 
\[H^0(D,\Omega^1(nD)_{|D})\to H^0\left(D, \Omega^1_X(nD)_{|D}\otimes_{\mathcal{O}_D} \underline{\Hom}_{{\rm Sh}_D}(\mathcal{O}_D, H^1_{\et,D})\right)\]
induced by the natural map $\mathcal{O}_D\to H^1_{\et,D}$ stemming from the Artin-Schreier sequence, 
and yields a global version of the characteristic form of
Matsuda-Yatagawa (which is a non-log version of the refined Swan conductor of Kato), see \cite[Sec. 5]{RS-AS} for details.
\item The characteristic form for differential forms is computed in \cite[Thm 6.6, 6.8]{RS-AS}. 
These computations are also used to prove the formula in Theorem \ref{thm:Omega} in positive characteristic.
\item It is an intriguing problem  to give a general (motivic) description of the image of  ${\rm char}^{(nD)}_F$.
For example, the images  of the characteristic forms  of differential forms and Witt vectors of finite length  
in positive characteristic are rather complicated and do not give a direct hint towards a general formula.
\end{itemize}

Finally, we give an exemplary application of how  a local form of Theorem \ref{thm:AS2}  
reveals an interesting structure of Chow groups of zero-cycles with modulus over local fields of equicharacteristic.

\begin{example}
Let $Y$ be a proper $k$-scheme with an effective Cartier divisor $E$, such that $V=Y\setminus E$ is smooth. 
By \cite[Corollary 2.3.5]{KSY-RecII} and \cite[Theorem 0.1]{Sai}  the Nisnevich sheafification 
of $\underline{\omega}_!h_0^{\boks}(Y,E)$ (see Definition \ref{def:ModulusCondition}) 
is a reciprocity sheaf $h_0(Y,E)_{\Nis}$. For a field $K$ over $k$ we have
\[h_0(Y,E)_{\Nis}(K)=CH_0(Y_K,E_K),\]
where the right hand side denotes the Chow group of zero-cycles with modulus and $Y_K=Y\otimes_k K$.

Assume $L$ is a henselian discrete valuation field of geometric type over $k$ with ring of integers $\mathcal{O}_L$, 
maximal ideal $\mathfrak{m}_L$, and residue field $K=\mathcal{O}_L/\mathfrak{m}_L$. 
For simplicity assume  the transcendence degree of $L/k$ is $1$, so that all 
geometric models of $(\mathcal{O}_L, \mathfrak{m}_L)$ have a projective SNC-compactification and 
$\Omega^1_{\mathcal{O}_L/k}\otimes_{\mathcal{O}_L} K\cong K$.
Then ${\rm fil}_n := h_0(Y, E)_{\Nis}(S, n s)$, where $S=\Spec \mathcal{O}_L$ and $s\in S$ is the closed point, defines a filtration
\[{\rm fil}_0 \subset {\rm fil}_1 \subset\ldots\subset {\rm fil}_n \subset \ldots\subset CH_0(Y_L, E_L),\]
where ${\rm fil}_0$ is the subgroup of $CH_0(Y_L, E_L)$ generated by closed points in $V_L$ whose closure in $V\times_k S$ is finite over $S$. 
By  Theorem \ref{thm:AS2} we have an injection (which depends on the choice of a local parameter\footnote{More precisely, we use the isomorphism
$\mathfrak{m}^{-n}\Omega^1_{\mathcal{O}_L}\otimes_{\mathcal{O}_L} K= \mathfrak{m^{-n}}/\mathfrak{m^{-n+1}}\cong K$})
\[{\rm fil}_n /{\rm fil}_{n-1} \hookrightarrow
\underline{\Hom}_{{\rm Sh}_K}(\mathcal{O}_K, h_0(Y,E)_{\Nis})(K) \qquad (n\ge 2).\]
However the internal hom on the right is not well understood
and it would be interesting to describe the image of the map.
Evaluating at 1 induces a canonical map
\[{\rm fil}_n/{\rm fil}_{n-1}\to  CH_0(Y_K,E_K)\qquad (n\ge 2).\]
This is a new map and it is tempting to view it as a specialization map (depending on the choice of a local parameter). 
It remains to study its properties  more closely,  e.g., if it happens to be injective for certain pairs $(Y,E)$.
\end{example}

\printbibliography

\noindent\textsc{Fakult\"{a}t Mathematik und Naturwissenschaften, Bergische Universit\"{a}t Wuppertal, Gau{\ss}stra{\ss}e 20, 42119 Wuppertal, Germany}\\
\textit{Email adress: }
\texttt{ruelling@uni-wuppertal.de} \\
\textit{URL: } \url{http://www2.math.uni-wuppertal.de/~ruelling/}

\vspace{5mm}

\noindent\textsc{Department of Mathematics, University of Oslo, Moltke Moes vei 35, 0851 Oslo, Norway}\\
\textit{Email adress: }
\texttt{ntmarti@math.uio.no} \\
\textit{URL: } \url{http://nikolaiopdan.com}

\end{document}